\documentclass{mcom-l}
\usepackage{url}
\usepackage{booktabs}
\usepackage{amssymb}
\usepackage{youngtab}
\usepackage{qtree}
\copyrightinfo{2009}{by the author}

\newcommand{\Z}{\mathbb{Z}}
\newcommand{\Zpk}[1]{\mathbb{Z}/p^{#1}\mathbb{Z}}
\newcommand{\Zp}{\mathbb{Z}/p\mathbb{Z}}

\newcommand{\identity}{1_{\scriptscriptstyle{G}}}

\newcommand{\DL}{{\rm DL}}
\newcommand{\EDL}{{\rm EDL}}
\newcommand{\evec}[1]{{\boldsymbol{#1}}}
\newcommand{\gvec}[1]{{\boldsymbol{#1}}}

\newcommand{\ZZ}{\mathbb{Z}}
\newcommand{\TDL}{T_{\scriptscriptstyle{\rm{DL}}}}
\newcommand{\TB}{T_{\scriptscriptstyle{\rm{B}}}}

\newtheorem{theorem}{Theorem}
\newtheorem{proposition}{Proposition}
\newtheorem{corollary}{Corollary}
\newtheorem{lemma}{Lemma}
\newtheorem{algorithm}{Algorithm}

\theoremstyle{definition}

\newtheorem{example}{Example}

\theoremstyle{remark}

\begin{document}

\title[Structure computation and discrete logarithms]{Structure computation and discrete logarithms in finite abelian $p$-groups}
\author{Andrew V. Sutherland}
\address{Massachusetts Institute of Technology}
\email{drew@math.mit.edu}
\subjclass[2000]{Primary 11Y16; Secondary 20K01, 12Y05}

\begin{abstract}
We present a generic algorithm for computing discrete logarithms in a finite abelian $p$-group $H$, improving the Pohlig--Hellman algorithm and its generalization to noncyclic groups by Teske.  We then give a direct method to compute a basis for $H$ without using a relation matrix.  The problem of computing a basis for some or all of the Sylow $p$-subgroups of an arbitrary finite abelian group $G$ is addressed, yielding a Monte Carlo algorithm to compute the structure of $G$ using $O(|G|^{1/2})$ group operations.  These results also improve generic algorithms for extracting $p$th roots in $G$.
\end{abstract}

\maketitle

\section{Introduction}

The discrete logarithm plays two opposing roles in group computations.  As a constructive tool, discrete logarithms are the key ingredient in generic algorithms for extracting roots (including square roots in finite fields) \cite{Adleman:Roots,Shanks:FiveAlgorithms,Sutherland:AbelianRoots,Tonelli:SquareRoot} and for computing group structure \cite{Buchmann:BabyGiant,Buchmann:GroupStructure,Sutherland:Thesis,Teske:GroupStructure,Teske:PohligHellmanStructure}.
On the other hand, a wide range of cryptographic applications depend on the essential difficulty of computing discrete logarithms in the worst case (see \cite{McCurly:DiscreteLogarithm} or \cite{Odlyzko:DiscretLogarithms} for a survey).

Typically, the discrete logarithm is defined in the context of a cyclic group: for any $\beta\in\langle\alpha\rangle$ there is a unique nonnegative integer $x < |\alpha|$ for which $\beta=\alpha^x$.  More generally, given $\gvec{\alpha}=(\alpha_1,\ldots,\alpha_r)$, if every $\beta\in\langle\gvec{\alpha}\rangle$ can be written uniquely as\footnote{Consistent with our use of the word ``logarithm", we write groups multiplicatively.}
$$\beta=\gvec{\alpha}^{\evec{x}}=\alpha_1^{x_1}\cdots\alpha_r^{x_r},$$
with $0\le x_i<|\alpha_i|$, then $\evec{x}=\DL(\gvec{\alpha},\beta)$ is the \emph{discrete logarithm} of~$\beta$ with respect to~$\gvec{\alpha}$, and we call the vector $\gvec{\alpha}$ a \emph{basis} for the group it generates.
%\footnote{This definition applies to any finite polycyclic group, but we treat only abelian groups here.}
We work in the computational framework of generic group algorithms, as defined, for example, in~\cite{Sutherland:Thesis}.  Thus we suppose that a ``black box" is used to perform group operations, possibly including the provision of random elements, with each group element arbitrarily assigned a unique identifier.

We are interested in constructive applications of the discrete logarithm, but let us first recall the negative result of Shoup \cite{Shoup:DLLowerBound}.  Any generic algorithm to compute discrete logarithms in a finite abelian group $G$ with prime exponent\footnote{The exponent of $G$ is the least positive integer $n$ for which $\alpha^n=\identity$ for all $\alpha\in G$.} uses $\Omega(|G|^{1/2})$ group operations.  A matching upper bound is achieved, for cyclic groups, by Shanks' baby-step giant-step algorithm \cite{Shanks:BabyGiant} and (probabilistically) by Pollard's rho method \cite{Pollard:RhoDL,Teske:SpeedingRho}.  Both algorithms can be generalized to compute discrete logarithms in any finite abelian group using $O(|G|^{1/2})$ group operations \cite{Buchmann:BabyGiant,Sutherland:Thesis,Teske:GroupStructure}.

However, when the exponent of the group is not prime, we can do better.  This was proven for cyclic groups by Pohlig and Hellman \cite{Pohlig:DiscreteLog} and later generalized by Teske \cite{Teske:PohligHellmanStructure}.\footnote{Pohlig and Hellman credit Roland Silver, and also Richard Schroeppel and H. Block, for (unpublished) independent discovery of the same algorithm \cite[p.~107]{Pohlig:DiscreteLog}.}  The Pohlig--Hellman approach relies on computing discrete logarithms in subgroups of the given group.  The reduction to subgroups of prime-power order is straightforward, hence we focus primarily on abelian $p$-groups.

If $\gvec{\alpha}$ is a basis for a finite abelian group $G$ of exponent $p^m$ and rank $r$, Teske's generalization of the Pohlig--Hellman algorithm computes $\DL(\gvec{\alpha},\beta)$ using
\begin{equation}\label{equation:TeskeBound}
\TDL(G) = O(m\lg|G|+mp^{r/2})
\end{equation}
group operations \cite[Thm.~6.1]{Teske:PohligHellmanStructure}.\footnote{Teske actually addresses a more general problem, find the minimal nontrivial solution $(\evec{x},y)$ to $\beta^y=\gvec{\alpha}^{\evec{x}}$, which we consider in Section \ref{section:Structure}.  Note that we use $\lg{x} = \log_2{x}$ throughout.}   When $m=1$ this reduces to the $O(|G|^{1/2})$ upper bound mentioned above.  If $p$ and $r$ are small (when computing square roots in finite fields, for example, $r=1$ and $p=2$) the first term dominates and the complexity becomes $O(n^2)$, where $n=\lg|G|$.  For cyclic groups this can be improved to $O(n\lg{n})$ \cite[\S 11.2.3]{Shoup:NumberTheoryAlgebra}, and here we achieve an $O(n\lg{n}/\lg\lg{n})$ bound for arbitrary finite abelian groups when $p$ and $r$ are suitably bounded.
More generally, Algorithm \ref{algorithm:DLp} computes $\DL(\gvec{\alpha},\beta)$ using
\begin{equation}
\TDL(G) = O\left(\frac{\lg(m+1)}{\lg\lg(m+2)}\lg|G|+\frac{\log_p|G|}{r}p^{r/2}\right)
\end{equation}
group operations, improving the dependence on $m$ in both terms of (\ref{equation:TeskeBound}).

Discrete logarithms may be applied to compute the structure of a finite abelian group.  Typically, one uses discrete logarithms to construct a relation matrix, which is then reduced to yield a basis by computing the Smith normal form \cite{Buchmann:BabyGiant,Buchmann:GroupStructure,Teske:GroupStructure}.  We take a simpler (and faster) approach, using our algorithm for discrete logarithms to directly construct a basis.  Given a generating set $S$ for a finite abelian $p$-group $G$ of rank $r$, we give a deterministic generic algorithm to construct a basis using
\begin{equation}\label{equation:TBbound}
\TB(S)=O\left(\lg^{2+\epsilon}|G|+(|S|-r+1)\TDL(G)\right)
\end{equation}
group operations, improving the $O(|S||G|^{1/2})$ result of Buchmann and Schmidt \cite{Buchmann:GroupStructure}.

The bound in (\ref{equation:TBbound}) is minimized when $|S|\approx r$.  If we pick a random subset $S\subset G$, of size $r+O(1)$, then $S$ generates $G$ with very high probability \cite{Pomerance:GeneratingAbelianGroups}.  When combined with an algorithm to compute the group exponent, this yields a generic Monte Carlo algorithm to compute the structure of an arbitrary finite abelian group using $O(|G|^{1/2})$ operations.  When sufficiently tight bounds on the group order are known, this can be converted to a Las Vegas algorithm.

This approach can also be applied to a Sylow $p$-subgroup $H\subset G$.  If the group exponent (or order) is known, the complexity then depends primarily on the size and shape of $H$, not $G$.  This is useful when extracting $p$th roots in $G$, which only requires a basis for $H$ \cite{Sutherland:AbelianRoots}.

\section{Abelian $p$-groups and Young tableaux}\label{section:Tableaux}

We begin by describing a bijection between finite abelian $p$-groups and Young tableaux that motivates our approach and allows us to fix some terminology.

We work in this section and the next with a basis $\gvec{\alpha}=(\alpha_1,\ldots,\alpha_r)$ for an abelian $p$-group $G$ of order $p^n$, exponent $p^m$, and $p$-rank (rank) $r$.  We let $|\alpha_i|=p^{n_i}$ and assume that $m=n_1\ge\cdots\ge n_r\ge 1$.  Up to isomorphism, $G$ is determined by the integer partition $\pi(G)=(n_1,\ldots,n_r)$.  For example, if
\begin{equation}\label{equation:ExampleGroup}
G\cong\Zpk{5}\times\Zpk{3}\times\Zp,
\end{equation}
then $\pi(G)=(5,3,1)$ is a partition of $n=9$ into three parts, with Young diagram:
$$\yng(5,3,1)$$
A cyclic group has $r=1$ and a single row in its diagram, while a group with prime exponent has $m=1$ and a single column.  In our example, $G$ has $r=3$ and $m=5$.

For each $\beta\in G$, we regard $\gvec{x}=\DL(\gvec{\alpha},\beta)$ as an element of the ring
\begin{equation}\label{equation:Ralpha}
R_{\gvec{\alpha}} = R_{\alpha_1}\times\cdots\times R_{\alpha_r} = \Zpk{n_1}\times\cdots\times\Zpk{n_r}.
\end{equation}
The additive group of $R_{\gvec{\alpha}}$ is isomorphic to $G$, via the map $\evec{x}\rightsquigarrow\gvec{\alpha}^{\evec{x}}$ (the inverse map sends $\beta$ to $\DL(\gvec{\alpha},\beta)$).  We may write the components of $\evec{x}\in R_{\gvec{\alpha}}$ in base $p$ as
\begin{equation*}
x_i = \sum_{j=1}^{n_i}p^{n_i-j}x_{i,j},
\end{equation*}
where $x_{i,1}$ is the \emph{most} significant digit (and may be zero).  We can then represent~$\evec{x}$ (and $\beta=\gvec{\alpha}^{\evec{x}}$) by a Young tableau of shape $\pi(G)$ with label $x_{i,j}$ in the $i$th row and $j$th column.  For our example $G$ in (\ref{equation:ExampleGroup}), if $p=2$ and $\evec{x}=(13,5,1)$, we have
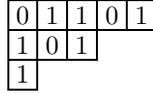
\newcommand{\bzero}{{\bf 0}}
\newcommand{\bone}{{\bf 1}}
\begin{equation}\label{example:start}
\young(01101,101,1)
\end{equation}
corresponding to $\beta=\gvec{\alpha}^{\evec{x}}=\alpha_1^{13}\alpha_2^{5}\alpha_3$.

We wish to split the tableau above into left and right halves, allowing us to write 
\begin{equation*}
\evec{x}=\evec{q}\evec{v}+\evec{u}.
\end{equation*}
The vector $\evec{q}$ is a ``shift'' vector whose components are powers of $p$, while $\evec{v}$ and $\evec{u}$ correspond to the left and right halves of $\evec{x}$, respectively. These vectors are obtained by computing discrete logarithms in certain subgroups of $G$, as we now describe.

If we multiply $\evec{x}$ (exponentiate $\beta$) by the integer scalar $p^k$, this shifts the labels of the tableau to the left $k$ places, leaving zeros on the right.  In our example, if $k=2$, we have $4\evec{x}=(20,4,0)$, yielding
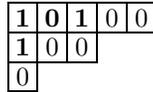
\begin{equation}\label{example:shift}
\young(\bone\bzero\bone 00,\bone 00,0)
\end{equation}
(with shifted labels in bold), corresponding to $\beta^{4}=\alpha_1^{20}\alpha_2^4$.  The element $\beta^{p^k}$ lies in the subgroup of $p^k$th powers in $G$,
\begin{equation}\label{equation:Gsubk}
G^{p^k}=\{\beta^{p^k}:\beta\in G\},
\end{equation}
which has a basis\footnote{Our definition of a basis allows $\gvec{\gamma}$ to contain trivial elements.  In practice we may truncate $\gvec{\gamma}$.} $\gvec{\gamma}$ defined by $\gamma_i=\alpha_i^{p^k}$.  The diagram of $G^{p^k}$ corresponds to the $m-k$ rightmost columns in the diagram of $G$.  In our example we have the shape $\pi(G^{4})=(3,1,0)=(3,1)$.  

Now let $\evec{u}=\DL(\gvec{\gamma},\beta^{p^k})$.  The vector $\evec{u}$ is an element of $R_{\gvec{\gamma}}$, but as a vector of integers written in base $p$, each component of $\evec{u}$ contains the low order $m-k$ digits of the corresponding component of $\evec{x}$.  We may ``clear" these digits of $\evec{x}$ to obtain $\evec{z}\in R_{\gvec{\alpha}}$ by subtracting $\evec{u}$ from $\evec{x}$ (in $\mathbb{Z}^r$), to obtain a reduced element of $R_{\gvec{\alpha}}$.  In our example we have $\evec{u}=(5,1,0)$, $\evec{z}=(8,4,1)$ and the tableau 
\begin{equation}\label{example:clear}
\young(\bzero\bone 000,\bone\bzero 0,\bone)
\end{equation}
with the entries unaffected by subtracting $\evec{u}$ from $\evec{x}$ in bold.  The element $\beta\gvec{\alpha}^{-\evec{u}}$ has order at most $p^k$ and lies in the $p^k$-torsion subgroup
\begin{equation}\label{equation:Gsupk}
G[p^k]=\{\beta: \beta^{p^k}=\identity,\thickspace \beta\in G\}.
\end{equation}
A basis $\gvec{\delta}$ for $G[p^k]$ is given by $\delta_i=\alpha_i^{q_i}$, where $q_i=p^{\max(0,n_i-k)}$.  The diagram of $G[p^k]$ corresponds to the $k$ leftmost columns of the diagram of $G$.  In our example, $\pi(G[4])=(2,2,1)$.
If we now let $\evec{v}=\DL(\gvec{\delta},\beta\gvec{\alpha}^{-\evec{u}})$, then $\evec{z}=\evec{q}\evec{v}$ and
\begin{equation*}
\evec{x}=\evec{q}\evec{v}+\evec{u},
\end{equation*}
as desired.  In our example we have $\evec{q}=(8,2,1)$ and $\evec{v}=(1,2,1)$ yielding
\begin{equation*}
(13,5,1) = (8,2,1)(1,2,1)+(5,1,0).
\end{equation*}
This equation effectively reconstructs the tableau in $(\ref{example:start})$ by gluing together the bold portions of the tableaux in $(\ref{example:shift})$ and $(\ref{example:clear})$.

Note that $\evec{u}$ and $\evec{v}$ were defined via discrete logarithms in the subgroups $G^{p^k}$ and $G[p^k]$, respectively.  This suggests a recursive approach, leading to base cases in subgroups corresponding to single columns in the Young diagram of $G$.
%   These base cases may be computed using a standard $O(p^{r/2})$ discrete logarithm computation (baby-steps giant-steps or Pollard rho) or, when $p^r$ is small, by precomputing a lookup table for $G^1$ (each base case will occur in a subgroup of $G^1$).
%We now formalize the discussion above and introduce a number of refinements.

\section{Computing discrete logarithms}

A recursive algorithm along the lines suggested above already yields an improvement over the result of
Teske \cite{Teske:GroupStructure}; in the cyclic case this is equivalent to Shoup's balanced divide-and-conquer version of the Pohlig--Hellman algorithm \cite[11.2.3]{Shoup:NumberTheoryAlgebra}.  We can achieve a further speedup by broadening the recursion tree, allowing us to take advantage of fixed-base exponentiation techniques.  At the same time, we can structure the algorithm to facilitate precomputation, an important practical optimization in applications that rely heavily on discrete logarithms \cite{Bernstein:SquareRoot,Sutherland:AbelianRoots}.

We will need to compute discrete logarithms in various subgroups of the form
\begin{equation}
G(j,k) = \{\beta^{p^j}:\beta^{p^k}=\identity,\thickspace \beta\in G\},
\end{equation}
for nonnegative integers $j<k$.  The subgroup $G(j,k)$ consists of all $p^j$th powers of order at most $p^{k-j}$ and
corresponds to columns $j+1$ through $k$ in the diagram of~$G$.  If $G$ has exponent $p^m$ then $G=G(0,m)$.

We wish to obtain a basis for $G(j,k)$ from our given basis $\gvec{\alpha}=(\alpha_1,\ldots,\alpha_r)$ for~$G$.  To this end, let $n_i=\log_p|\alpha_i|$, let $q_i=p^{j+\max(0,n_i-k)}$, and define
\begin{equation}\label{equation:alphajk}
\evec{q}(j,k)=(q_1,\ldots,q_r)\qquad\text{and}\qquad\gvec{\alpha}(j,k)=\gvec{\alpha}^{\evec{g}(j,k)} = (\alpha^{q_1},\ldots,\alpha^{q_r}).
\end{equation}
%We also define $\gvec{q}(j,k)=(q_1,\ldots,q_r)$.
Then $\gvec{\alpha}(j,k)$ is our desired basis, as we now show.

\begin{lemma}\label{lemma:BasisVector}
Let $\gvec{\alpha}=(\alpha_1,\ldots,\alpha_r)$ be a basis for a finite abelian $G$, and let $j$ and $k$ be nonnegative integers with $j<k$.
Then $\gvec{\alpha}(j,k)$ is a basis for $G(j,k)$.
\end{lemma}
\begin{proof}
Let $\gvec{\gamma} = \gvec{\alpha}(j,k)$.  
We first show that $\gamma$ is a basis for $\langle\gvec{\gamma}\rangle$.
Suppose for the sake of contradiction that $\gvec{\gamma}^{\evec{x}} = \gvec{\gamma}^{\evec{y}}$ with $\evec{x},\evec{y}\in R_{\gvec{\gamma}}$ distinct.  We must have $x_i\ne y_i$ for some $i$, which implies $\gamma_i\ne \identity$ and $q_i|\gamma_i|=|\alpha_i|$.  We also have $\gvec{\alpha}^{\evec{q}(j,k)\evec{x}}=\gvec{\alpha}^{\evec{q}(j,k)\evec{y}}$.  As $q_ix_i$ and $q_iy_i$ are distinct integers less than $|\alpha_i|$, the vectors $\evec{q}(j,k)\evec{x}$ and $\evec{q}(j,q)\evec{y}$ are distinct elements of $R_{\gvec{\alpha}}$.  But this is a contradiction, since $\gvec{\alpha}$ is a basis.

We now prove $\langle\gvec{\gamma}\rangle=G(j,k)$.
Every $\gamma_i$ is a $p^j$th power (since $p^j|q_i$) and has order at most $p^{k-j}$ (since $p^{n_i}|p^{k-j}q_i$); thus $\langle\gvec{\gamma}\rangle\subset G(j,k)$.  Conversely, for $\delta\in G(j,k)$, if $\evec{x}=\DL(\gvec{\alpha},\delta)$, then $p^j|x_i$ and $p^{n_i}|p^{k-j}x_i$ for each $i$.  If $k<n_i$, then $p^{j+n_i-k}|x_i$, so $q_i|x_i$ in every case.  It follows that $\delta\in\langle\gvec{\gamma}\rangle$; hence $G(j,k)\subset \langle\gvec{\gamma}\rangle$.
\end{proof}

We now give a recursive algorithm to compute discrete logarithms in $G(j,k)$, using $\gvec{\alpha}(j,k)$ and $\evec{q}(j,k)$ as defined above.  Note that if $j\le j'<k'\le k$, then each component of $\evec{q}(j,k)$ divides the corresponding component of $\evec{q}(j',k')$, and we may then write $\evec{q}(j',k')/\evec{q}(j,k)$ to denote point-wise division.  For convenience, let $\DL_{\gvec{\alpha}}(j,k,\beta)$ denote $\DL(\gvec{\alpha}(j,k),\beta)$.  We assume the availability of a standard algorithm for computing discrete logarithms in the base cases, as discussed below.

\begin{algorithm}\label{algorithm:DLp}
Given a basis $\gvec{\alpha}$ for a finite abelian $p$-group $G$ and $t\in\Z_{>0}$, the following algorithm computes $\DL_{\gvec{\alpha}}(j,k,\beta)$ for integers $0\le j < k$ and $\beta\in G(j,k)$:
\end{algorithm}
\renewcommand\labelenumi{\theenumi.}
\begin{enumerate}
\item
If $k-j\le t$, compute $\evec{x}\leftarrow\DL_{\gvec{\alpha}}(j,k,\beta)$ as a base case and return $\evec{x}$.
\vspace{4pt}
\item
Choose integers $j_1,\ldots,j_w$ satisfying $j=j_1<j_2<\cdots<j_w<j_{w+1}=k$.
\vspace{4pt}
\item
Compute $\gamma_i=\beta^{p^{{j_i}-j}}$ for $i$ from 1 to $w$, and set $\evec{x}\leftarrow 0$.
\vspace{4pt}
\item
For $i$ from $w$ down to 1:
\renewcommand\labelenumii{\theenumii.}
\begin{enumerate}
\item
\vspace{4pt}
Recursively compute $\evec{v}\leftarrow \DL_{\gvec{\alpha}}(j_i,j_{i+1},\gamma_i\gvec{\alpha}(j_i,k)^{-\evec{x}})$.
\item
\vspace{4pt}
Set $\evec{x}\leftarrow \evec{s}\evec{v} + \evec{x}$, where $\evec{s}=\evec{q}(j_i,j_{i+1})/\evec{q}(j_i,k)$.
\end{enumerate}
\vspace{4pt}
\item 
Return $\evec{x}$.
\end{enumerate}

\begin{example}
Let $G$ be cyclic of order $p^{19}$ with basis $\alpha$, $j=6$ and $k=13$.  Then $q(6,13)=p^{6+\max(0,19-13)}=p^{12}$ and $\alpha^{p^{12}}$ is a basis for $G(6,13)$.  Let $j_2=8$ and $j_3=11$, so that $(6,13]$ is partitioned into subintervals $(6,8]$, $(8,11]$, and $(11,13]$.  We then have $q(11,13)=p^{17}$, $q(8,11)=p^{16}$, $q(6,8)=p^{17}$, and also $q(8,13)=p^{14}$.
For $\beta\in G(6,13)$, Algorithm \ref{algorithm:DLp} computes
\begin{tabbing}
\hspace{50pt}\=$v_3 = \DL(\alpha^{p^{17}},\beta^{p^5}),$\hspace{80pt}\=$x_3 = v_3,$\\
\>$v_2 = \DL(\alpha^{p^{16}},\beta^{p^2}\alpha^{-p^{14}x_3}),$\>$x_2 = p^2v_2+v_3,$\\
\>$v_1 = \DL(\alpha^{p^{17}},\beta\alpha^{-p^{12}x_2}),$\>$x_1 = p^5v_1+p^2v_2+v_3.$
\end{tabbing}
The final value $x=x_1$ contains 7 base $p$ digits: 2 in $v_1$, 3 in $v_2$, and 2 in $v_3$.
\end{example}
\begin{example}
Suppose instead that $G$ is cyclic of order $p^{9}$, but keep the other parameters as above.
We then have $q(6,13)=p^6$, $q(11,13)=p^{11}$, $q(8,11)=p^8$, $q(6,8)=p^7$, and $q(8,13)=p^8$.
For $\beta\in G(6,13)$, the algorithm now computes
\begin{tabbing}
\hspace{50pt}\=$v_3 = \DL(\identity,\identity),$\hspace{92pt}\=$x_3 = v_3 = 0,$\\
\>$v_2 = \DL(\alpha^{p^{8}},\beta^{p^2}\alpha^{-p^8x_3}),$\>$x_2 = v_2,$\\
\>$v_1 = \DL(\alpha^{p^{7}},\beta\alpha^{-p^6x_2}),$\>$x_1 = pv_1+v_2.$
\end{tabbing}
The computation of $x_3$ requires no group operations; the algorithm can determine $\alpha(11,13)=\identity$ from the fact that $11\ge 9$ (since $|\alpha|=p^9$ is given).  The final value $x=x_1$ contains 3 base $p$ digits: 2 in $v_1$ and 1 in $v_2$.
\end{example}

These examples illustrate the general situation; we compute discrete logarithms in $r$ cyclic groups in parallel.  The second example is contrived, but it shows what happens when a cyclic factor of $G$ has order less than $p^k$.

We assume that no cost is incurred by trivial operations (those involving the identity element).  As a practical optimization, the loop in step 4 may begin with the largest $i$ for which $\gamma_i\ne \identity$ (it will compute $\evec{x}=0$ up to this point in any event).

The correctness of Algorithm \ref{algorithm:DLp} follows inductively from the lemma below.

\begin{lemma}\label{lemma:DLpCorrectness}
Let $\gvec{\alpha}$ be a basis for a finite abelian $p$-group $G$ and let $j,j',k',$ and $k$ be integers with $0\le j\le j' < k' \le k$.  For all $\beta\in G(j,k)$ the following hold:
\renewcommand\labelenumi{(\roman{enumi})}
\begin{enumerate}
\item
If $\evec{x}=\DL_{\gvec{\alpha}}(k',k,\beta^{p^{k'-j}})$ and $\gamma=\beta^{p^{j'-j}}\gvec{\alpha}(j',k)^{-\evec{x}}$, then $\gamma\in G(j',k')$.
\vspace{4pt}
\item
If we also have $\evec{v}=\DL_{\gvec{\alpha}}(j',k',\gamma)$ and $\evec{s}=\evec{q}(j',k')/\evec{q}(j',k)$,\\then $\evec{s}\evec{v} + \evec{x}=\DL_{\gvec{\alpha}}(j',k,\beta^{p^{j'-j}})$.
\end{enumerate}
\end{lemma}
\begin{proof}
For (i), note that $\beta$ is a $p^j$th power, so $\beta^{p^{j'-j}}$ is a $p^{j'}$th power, and every element of $\langle\gvec{\alpha}(j',k)\rangle$ is a $p^{j'}$th power, hence $\gamma$ is a $p^{j'}$th power.  We also have
\begin{equation*}
\gamma^{p^{k'-j'}} = \left(\beta^{p^{j'-j}}\gvec{\alpha}(j',k)^{-\evec{x}}\right)^{p^{k'-j'}}
= \beta^{p^{k'-j}}\gvec{\alpha}(j',k)^{-p^{k'-j'}\evec{x}}.
\end{equation*}
It follows from the definition in (\ref{equation:alphajk}) that $\gvec{\alpha}(j',k)^{-p^{k'-j'}\evec{x}}=\gvec{\alpha}(k',k)^{-\evec{x}}$, since we have $j'+\max(0,n_i-k) + k' - j'= k' + \max(0,n_i-k)$.  We then obtain
$$\gamma^{p^{k'-j'}}=\beta^{p^{k'-j}}\gvec{\alpha}(k',k)^{-\evec{x}} = \beta^{p^{k'-j}}(\beta^{p^{k'-j}})^{-1} = \identity.$$
Thus $\gamma$ has order at most $p^{k'-j'}$, and therefore $\gamma\in G(j',k')$, proving (i).

Note that $k'<k$ implies $\max(0,n_i-k')\ge\max(0,n_i-k)$, so $\evec{q}(j',k')$ is divisible (component-wise) by $\evec{q}(j',k)$, and $\evec{s}$ in (ii) is well defined.  Now
\begin{equation*}
\gvec{\alpha}(j',k)^{\evec{s}\evec{v}}
= \gvec{\alpha}^{\evec{q}(j',k)\evec{s}\evec{v}} = \gvec{\alpha}^{\evec{q}(j',k')\evec{v}}
= \gvec{\alpha}(j',k')^{\evec{v}} = \gamma = \beta^{p^{j'-j}}\gvec{\alpha}(j',k)^{-\evec{x}},
\end{equation*}
and therefore $\gvec{\alpha}(j',k)^{\evec{s}\evec{v}+\evec{x}}=\beta^{p^{j'-j}}$, proving (ii).
\end{proof}

We now consider the parameter $t$ in Algorithm \ref{algorithm:DLp}.  If $p^r$ is small, we precompute a lookup table for $G(0,t)$, containing at most $p^{rt}$ group elements, for some suitable value of $t$.  This will handle all the base cases, since they arise in subgroups $G(j,k)$ of $G(0,t)$, where $k-j\le t$.  This is especially effective when one can amortize the cost over many discrete logarithm computations, in which case a larger $t$ is beneficial.  In applications where $p^r=O(1)$, one typically chooses $t$ to be logarithmic in the relevant problem size (which may be larger than $|G|$).

When $p^r$ is large, we instead set $t=1$ and use a standard $O(\sqrt{N})$ algorithm for computing discrete logarithms in finite abelian groups.  A space-efficient algorithm derived from Pollard's rho method is given in \cite{Teske:GroupStructure}, and a baby-steps giant-steps variant can be found in \cite[Alg.~9.3]{Sutherland:Thesis} (see Section \ref{section:Performance} for optimizations).

When partitioning the interval $(j,k]$ into subintervals in step 2, we assume that the subintervals are of approximately equal size, as determined by the choice of $w$.  The choice $w=k-j$ limits the recursion depth to 1 and corresponds to the standard Pohlig--Hellman algorithm.  The choice $w=2$ yields a balanced binary recursion tree.  This might appear to be an optimal choice, but we can actually do better with a somewhat larger choice of $w$, using fixed-base exponentiation techniques.

We recall a theorem of Yao.
\begin{theorem}[Yao\footnote{Pippenger gives a better bound for large $w$, but not necessarily an online algorithm \cite{Bernstein:Pippenger,Pippenger:PowersPrelim}.}]\label{theorem:Yao}
There is an online algorithm that, given $\gamma\in G$ and any input sequence of positive integers $e_1,\ldots,e_w$, outputs $\gamma^{e_1},\ldots,\gamma^{e_w}$ using at most
$$\lg{E} + c\sum_{i=1}^w\left\lceil\frac{\lg{e_i}}{\lg\lg(e_i+2)}\right\rceil$$
multiplications, where $E=\max_i\{e_i\}$ and $c\le 2$ is a constant.\footnote{As $E\to\infty$ the constant $c$ can be made arbitrarily close to 1.}
\end{theorem}
The online algorithm in the theorem outputs $\gamma^{e_i}$ before receiving the input $e_{i+1}$. If we set $n=\lg{E}$, Yao's Theorem tells us that, provided $w=O(\lg{n})$, we can perform $w$ exponentiations of a common base with $n$-bit exponents using just $O(n)$ multiplications, the same bound as when $w=1$.  There are several algorithms that achieve Yao's bound \cite{Bernstein:Pippenger,BGMW:FastExponentiation,Gordon:Survey,Lim:Exponentiation,Menezes:Handbook}, and they typically require storage for $O(n/\lg{n})$ group elements.

Consider the execution of Algorithm 1 computing $\DL(\gvec{\alpha},\beta)=\DL(0,m,\beta)$.    It will be convenient to label the levels of the recursion tree with $\ell=0$ at the bottom and $\ell=d$ at the top, where $d$ is the maximum depth of the recursion.  At each level the interval $(0,m]$ is partitioned into successively smaller subintervals.  We let $s_d=m$ denote the size of the initial interval at level $d$, and at level $\ell$ we partition each interval into approximately $w_\ell$ subintervals of maximum size $s_{\ell-1}= \left\lceil s_\ell/w_\ell\right\rceil$
and minimum size $\left\lfloor s_\ell/w_\ell \right\rfloor$.

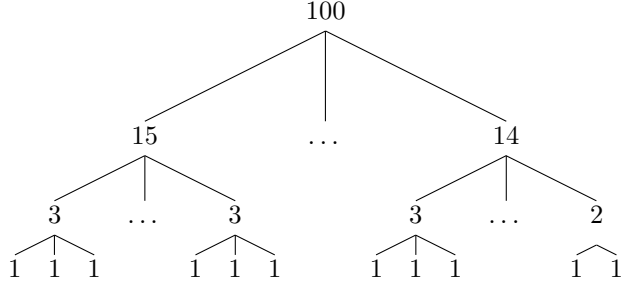
\vspace{18pt}
\begin{centering}
\Tree [.100 [.15 [.3 1 1 1 ].3 {\ldots} [.3 1 1 1 ].3 ].15 {\ldots} [.14 [.3 1 1 1 ].3 {\ldots} [.2 1 1 ].2 ].14 ]
\end{centering}
\vspace{18pt}

In the tree above we start at level $\ell=3$ with $s_3=m=100$ and $w_3=7$, partitioning 100 into two subintervals of size $s_2=15$ and five subintervals of size~14.  We then have $w_2=5$ and $s_1=3$, and finally $w_1=3$ and $s_0=1$.  The base cases are all at level 0 in this example, but in general may also occur at level 1.  The fan-out of each node at level $\ell$ is $w_\ell$, except possibly at level 1 (in this example, we cannot partition 2 into three parts).

Our strategy is to choose $w_\ell \approx \min(\lg(s_\ell\lg{p}),s_\ell)$ and apply Yao's Theorem to bound the cost at each level of the recursion tree by $O(\lg|G|)$ group operations, not including the base cases.
The standard Pohlig--Hellman approach reduces the problem to base cases in one level, potentially incurring a cost of $O(\lg^2|G|)$ to do so. A binary recursion uses $O(\lg|G|)$ group operations at each level, but requires $\Omega(\lg{m})$ levels, while we only need $O(\lg{m}/\lg\lg{m})$.

%Note that in general $G$ is not cyclic, but we may consider the cost of Algorithm \ref{algorithm:DLp} associated with each cyclic factor generated by some $\alpha_i$ in our basis.  A nice feature of the algorithm is that the cost associated with $\alpha_i$ depends essentially only on $|\alpha_i|$, not on the exponent $p^m$ of $G$.  This is because the relevant component $x_i$ of $\evec{x}$ in step 4 may be zero for much of the computation when $|\alpha_i|<p^m$.

With these ideas in mind, we now prove an absolute bound on the running time of Algorithm \ref{algorithm:DLp}.  An asymptotic bound appears in the corollary that follows.

\begin{proposition}\label{proposition:DLComplexity}
Let $\gvec{\alpha}=(\alpha_1,\ldots,\alpha_r)$ be a basis for a finite abelian $p$-group $G$ with rank $r$ and exponent $p^m$.  Set $n_i=\log_p|\alpha_i|$, and let $r_j$ be the rank of the subgroup of $p^j$th powers in $G$.  There is a generic algorithm to compute $\DL(\gvec{\alpha},\beta)$ using
$$\TDL(G) \le c\left(\sum_{i=1}^r\frac{\lg(n_i+1)}{\lg\lg(n_i+2)}\lg|\alpha_i|+\sum_{j=0}^{m-1} p^{r_j/2}\right)$$
group operations, where $c$ is an absolute constant independent of $G$.
\end{proposition}
When a probabilistic algorithm is used for the base cases (such as the rho method), the algorithm in the proposition is probabilistic and $\TDL(G)$ refers to the expected running time, but otherwise the algorithm is deterministic.
The bound on $\TDL(G)$ depends only on the structure of $G$, not the basis $\gvec{\alpha}$.

\begin{proof}[Proof of Proposition \ref{proposition:DLComplexity}]
We use Algorithm \ref{algorithm:DLp} to compute $\DL(\gvec{\alpha},\beta)=\DL_{\gvec{\alpha}}(0,m,\beta)$ using $t=1$.
As discussed above, we label the levels of the recursion tree with $\ell=0$ at the base and $\ell=d$ at the root.
We let $s_\ell$ denote the size of the first interval at level $\ell$, and we assume that all others have
size at least $s_\ell-1$ and at most~$s_\ell$.
Thus $s_d=m$, and we recursively define $s_{\ell-1}=\lceil s_\ell/w_\ell\rceil$ down to $s_0=1$.
To simplify the proof we use $w_\ell=\lceil \lg (2s_\ell)\rceil$ (independent of $p$) and assume $n_1=m$.

There is a base case in $G(j,j+1)$ for each $0\le j < m$, with $|G(j,j+1)|=p^{r_j}$.
Applying either of the standard $O(\sqrt{N})$ discrete logarithm algorithms to the base cases yields the second sum in the bound for $\TDL(G)$, with $c\approx 2$.

At level $\ell$ of the recursion tree, the total cost of step 3 is bounded by
\begin{equation*}
T_1(\ell)< m(2\lg{p})=2\lg|\alpha_1|,
\end{equation*}
since $m$ exponentiations by $p$ are required.  The total cost of the multiplications by~$\gamma_i$ in step 4a is bounded by $T_2(\ell) = m\le \lg|\alpha_1|$.

All other group operations occur in exponentiations in step 4a.  These are of the form $\gvec{\alpha}(j',k)^{-\evec{x}}$, where $j\le j'<k$.
To bound their cost, we consider the cost $T(\alpha_i,\ell)$ associated to a particular $\alpha_i$ at level $\ell$.
The exponentiation of $\alpha_i$ is nontrivial only when $j'<n_i$ and $x_i\ne 0$.
Thus to bound $T(\alpha_i,\ell)$, we only count exponentiations with $j'<n_i$ and only consider levels $\ell$ of the recursion tree with $s_{\ell-1} < n_i$, since at all higher levels $x_i$ will still be zero.

In the recursive call to compute $\DL_{\gvec{\alpha}}(j,k,\beta)$, we may compute $\gvec{\alpha}(j',k)^{-\evec{x}}$ using fixed bases $\alpha_i^{-q_i}$, where $\evec{q}=\evec{q}(j,k)$ as in (\ref{equation:alphajk}), since $\evec{q}(j,k)$ divides $\evec{q}(j',k)$ for all $j'\ge j$.  For each $\alpha_i$ we can precompute all the $\alpha_i^{-q_i}$ for a cost of
\begin{equation*}
T_0(\alpha_i)< 2n_i\lg{p} = 2\lg|\alpha_i|.
\end{equation*}
At level $\ell > 0$ with $s_{\ell-1} < n_i$, there are $\lceil n_i/s_l \rceil$ instances of up to $w_\ell-1$ nontrivial exponentiations involving $\alpha_i$.
These are computed using the common base $\alpha_i^{-q_i}$, with exponents bounded by $E=\min(p^{s_\ell},|\alpha_i|)$.
Applying Yao's Theorem,
\begin{equation*}
T(\alpha_i,\ell)\le \lceil n_i/s_l \rceil\left(\lg{E} + 2(w_\ell-1)\left\lceil\frac{\lg{E}}{\lg\lg(E+2)}\right\rceil\right).
\end{equation*}
If $s_\ell > n_i$, we replace $\lceil n_i/s_\ell\rceil$ by 1 and $\lg E$ by $\lg |\alpha_i|$; otherwise we replace $\lg E$ by $s_\ell\lg p$.
We then apply $\lceil z\rceil\le 2z$ (for $z\ge 1$) to remove both ceilings and obtain
\begin{equation}\label{equation:Tibnd1}
T(\alpha_i,\ell) \le 2\lg|\alpha_i| + 8\lg|\alpha_i|\left(\frac{w_\ell-1}{\lg\lg(E+2)}\right).
\end{equation}
If $E=p^{s_\ell}$, then $(w_\ell-1)/\lg\lg(E+2) < 2$.  Otherwise $E=|\alpha_i|$, and then
$$\lg(E+2) > n_i \ge s_{\ell-1} + 1 = \lceil s_\ell /w_\ell \rceil + 1 \ge s_\ell / (\lg s_\ell+2) + 1,$$
which implies $(w_\ell-1)/\lg\lg(E+2) < 3$.  This yields $T(\alpha_i,\ell) < 26\lg|\alpha_i|$.

It remains to bound the number of levels $\ell > 0$ with $s_{\ell-1} < n_i$.
This is equal to the least $\ell$ for which $s_\ell \ge n_i$, which we denote $d(n_i)$.
We derive an upper bound on $d(n_i)$ as a function of $n_i$ by proving a lower bound on $s_\ell$ as a function of~$\ell$.

We recall that $s_0=1$ and for $\ell>0$ we have $w_\ell=\lceil\lg(2s_\ell)\rceil\ge 2$ and $s_\ell \ge 2s_{\ell-1}$.
This implies $w_\ell\ge l+1$ and $s_\ell\ge \ell s_{\ell-1} \ge \ell!$ for all $\ell >0$.
Stirling's formula yields a lower bound on $s_\ell$, from which one obtains the upper bound
\begin{equation}\label{equation:dni}
d(n_i)\le \frac{2\lg(n_i+1)}{\lg\lg(n_i+2)},
\end{equation}
valid for $n_i\ge 14$.  The lexicographically minimal sequence of integers satisfying $s_0=1$ and $s_{\ell-1} = \lceil s_\ell/\lceil\lg(2s_\ell)\rceil\rceil$ for all $\ell > 0$ begins $1,2,4,16,121,1441,\ldots$, and one finds that $s_1,\ldots,s_{13}$ satisfy (\ref{equation:dni}), with $\ell=n_i$ and $s_\ell=d(n_i)$.

The total cost of all computations outside of the base cases is then bounded by
\begin{equation*}
\sum_{\ell=1}^d \Bigl(T_1(\ell)+T_2(\ell)\Bigr) + \sum_{i=1}^r\left(T_0(\alpha_i) + \sum_{\ell=1}^{d(n_i)} T(\alpha_i,\ell)\right)\medspace\le\medspace c\sum_{i=1}^r\frac{\lg(n_i+1)}{\lg\lg(n_i+2)}\lg|\alpha_i|,
\end{equation*}
where we use $d=d(n_1)$, and the constant $c<57$.  This yields the first sum in the bound for $\TDL(G)$ and completes the proof.
\end{proof}

For the sake of brevity we have overestimated the constant $c$ in the proof above.  Empirically, $c$ is always less than 2 and is typically close to 1 (see Section \ref{section:Performance}).

The space used by Algorithm 1 depends on how the base cases are handled, but can be bounded by $O(\lg|G|/\lg\lg|G|)$ group elements.  There are at most $\sum_{i=1}^rd(n_i)\sqrt{n_i}$ distinct $\alpha_i^{-q_i}$ that need to be precomputed, which fits within this bound.  In practice,  additional precomputation using slightly more storage, perhaps $O(\lg|G|)$ elements, can accelerate both the exponentiations and the base cases \cite{Sutherland:AbelianRoots}.

%  Note that Proposition \ref{proposition:DLComplexity} does not require us to make any of the parameters large, allowing us to give uniform asymptotic bounds as any of $p$, $m$, or $r$ tend toward infinity (or not).

For many groups arising ``in nature'', both sums in the bound for $\TDL(G)$ are typically dominated by their first terms.  Divisor class groups of curves and ideal class groups of number fields are, at least heuristically, two examples. One often sees an $L$-shaped Young diagram, with $n\approx m+r$, where $n=\log_p|G|$.  Algorithm~1 yields a useful improvement here, with a complexity of $O(p^{r/2})$ versus $O(mp^{r/2})$.  More generally, we have the following corollary.
\begin{corollary}\label{corollary:DLComplexity}
Let $\gvec{\alpha}$ be a basis for a finite abelian $p$-group $G$ of size $p^n$, exponent~$p^m$, and rank $r$.  There is a generic algorithm to compute $\DL(\gvec{\alpha},\beta)$ using
$$\TDL(G) = O\left(\frac{\lg(m+1)}{\lg\lg(m+2)}\lg|G|+\frac{n}{r}p^{r/2}\right)$$
group operations.
\end{corollary}
\begin{proof}
The first term is immediate from Proposition \ref{proposition:DLComplexity} since $\lg{G}=\sum_{i=1}^r\lg|\alpha_i|$ and $n_i\le m$ for all $i$.
For the second term, consider $\sum_{i=0}^{m-1}p^{r_i/2}$.  We have $r=r_0\ge r_i$ and $n=r_0+\cdots+r_{m-1}$.  For fixed $r$ and $n$ the worst case, up to a constant factor, occurs when the $r_i$ are roughly equal (one uses Lemma \ref{lemma:prodsum} to prove this).
\end{proof}

The asymptotic upper bound on $\TDL(G)$ achieved here is nearly tight for generic algorithms.
When $r=O(n)$, the bound in the corollary becomes $O(p^{r/2})$, matching Shoup's $\Omega(p^{r/2})$ lower bound.  Even when this is not the case, one may argue, along the lines of Shoup, that the sum $\sum p^{r_i/2}$ in Proposition \ref{proposition:DLComplexity} is tight in any event.  The first term in the corollary is $O(\lg^{1+\epsilon}|G|)$, and one does not expect to do better than $O(\lg|G|)$, due to the $\Omega(\lg|G|)$ lower bound for exponentiation.

To complete our discussion of discrete logarithms, we give an algorithm to compute $\DL(\gvec{\alpha},\beta)$ in an arbitrary finite abelian group $G$.  We assume that $\gvec{\alpha}$ is a prime-power basis for $G$, composed of bases $\gvec{\alpha}_p$ for each of the Sylow $p$-subgroups of $G$.  The construction of such a basis is discussed in the next section (and readily obtained from a given basis in any event).

\begin{algorithm}\label{algorithm:DL}
Given a prime-power basis $\gvec{\alpha}$ for a finite abelian group $G$ with $|G|=N=q_1\cdots q_k$ a factorization into powers of distinct primes $p_1, \ldots, p_k$, and $\beta\in G$, the following algorithm computes $\evec{x}=\DL(\gvec{\alpha},\beta)$:
\end{algorithm}
\renewcommand\labelenumi{\theenumi.}
\begin{enumerate}
\item
Let $M_j=N/q_j$ and compute $\beta_j\leftarrow\beta^{M_j}$ for $j$ from 1 to $k$.
\vspace{4pt}
\item
Compute $\evec{x}_j\leftarrow \DL(\gvec{\alpha}_{p_j},\beta_j)$ using Algorithm \ref{algorithm:DLp}.
\vspace{4pt}
\item
Set $\evec{x}\leftarrow\evec{x}_1/M_1\circ\cdots\circ\evec{x}_k/M_k$.
\end{enumerate}
\renewcommand\labelenumi{\theenumi}
\vspace{6pt}
The symbol ``$\circ$" denotes concatenation of vectors.  Since $\beta_j^{q_j}=\identity$, we must have $\beta_j\in\langle\gvec{\alpha}_{p_j}\rangle$, and the components of $\gvec{\alpha}_{p_j}$ all have order a power of $p_j$.  The exponent vector $\evec{x}_j$ is thus divisible by $M_j$, since $M_j$ is coprime to $p_j$ and therefore a unit in each factor of the ring $R_{\gvec{\alpha}_{p_j}}$.  The correctness of Algorithm \ref{algorithm:DL} follows easily.
\smallbreak

Let $n=\lg N$.  For $k=O(\lg{n})$, the exponentiations in step 1 can be performed using $O(n)$ group operations, by Yao's Theorem.
As $k$ approaches $n$, this bound increases to $O(n^2/(\lg n)^2)$, and one should instead apply the $O(n\lg n/\lg\lg n)$ algorithm of \cite[Alg.~7.4]{Sutherland:Thesis}.  The total running time is then
\begin{equation}\label{equation:TDLGeneral}
\TDL(G) = O\left(\frac{\lg(k+1)}{\lg\lg(k+2)}\lg|G|\right) + \sum_{j=1}^k \TDL(G_{p_j}),
\end{equation}
group operations, where $G_{p_j}$ denotes the Sylow $p_j$-subgroup of $G$.  For sufficiently large $|G|$, the bound for $\TDL(G)$ is dominated by the sum in (\ref{equation:TDLGeneral}).

\section{Constructing a basis for a finite abelian $p$-group}\label{section:Structure}

For a finite abelian group $G$, the {\em group structure} problem asks for a factor decomposition of $G$ into cyclic groups of prime-power order, with a generator for each factor.  This is equivalent to computing a basis for each of the (nontrivial) Sylow $p$-subgroups of $G$.  We first suppose that $G$ is a $p$-group and then give a reduction for the general case in Section \ref{section:GeneralBasis}.

Typically, a basis is derived from a matrix of relations among elements of a generating set for the group \cite{Buchmann:BabyGiant,Buchmann:GroupStructure,Buchmann:BinaryQuadraticForms,Teske:GroupStructure}.
This generating set may be given, or obtained (with high probability) from a random sample.
One then computes the Smith normal form of the relation matrix \cite[\S 2.4]{Cohen:CANT}, applying corresponding group operations to the generating set to produce a basis.

Relations may be obtained via \emph{extended discrete logarithms}.  If $\gvec{\alpha}$ is a basis for a subgroup of $G$  and $\beta\in G$, then $\EDL(\gvec{\alpha},\beta)$ is the pair $(\evec{x},y)$ satisfying $\beta^y=\gvec{\alpha}^{\evec{x}}$ that minimizes $y>0$, with $\evec{x}\in R_{\gvec{\alpha}}$. In a $p$-group, $y$ is necessarily a power of $p$.

While our approach does not require us to compute $\EDL(\gvec{\alpha},\beta)$, we note that any algorithm for $\DL(\gvec{\alpha},\beta)$ can be used to compute $\EDL(\gvec{\alpha},\beta)$.% with only a slight increase in the running time.
\begin{lemma}\label{lemma:EDLreduction}
Given a basis $\gvec{\alpha}$ for a subgroup of a finite abelian $p$-group $G$ and $\beta\in G$, there is a generic algorithm to compute $\EDL(\gvec{\alpha},\beta)$ using at most
$$\left\lceil\lg(\log_p|\beta|)\right\rceil \TDL(G) + 2\lg|\beta|$$
group operations.
\end{lemma}
\begin{proof}
Assume Algorithm \ref{algorithm:DLp} returns an error whenever a base case fails.\footnote{Failure detection with  baby-steps giant-steps or lookup table is easy.  See \cite{Teske:GroupStructure} for a rho search.}  Compute $\beta^{p^j}$ for $0 \le j \le \log_p|\beta|$.  Then use a binary search to find the least $j$ for which one can successfully compute $\evec{x}=\DL(\gvec{\alpha},\beta^{p^j})$.  We then have $\EDL(\gvec{\alpha},\beta)=(\evec{x},p^j)$.
\end{proof}

Teske gives an algorithm to directly compute $\EDL(\gvec{\alpha},\beta)$, avoiding the $\lg(\log_p|\beta|)$ factor above, but this may still be slower than applying Lemma \ref{lemma:EDLreduction} to Algorithm \ref{algorithm:DLp}.  Alternatively, we may modify Algorithm $\ref{algorithm:DLp}$ to solve a slightly easier problem.  Instead of solving $\beta^y=\gvec{\alpha}^{\evec{x}}$, we seek a solution to $\beta^y=\gvec{\alpha}^{y\evec{x}}$.

More specifically, let us define a function $\DL_{\gvec{\alpha}}^*(j,k,\beta)$ that extends the function $\DL_{\gvec{\alpha}}(j,k,\beta)$ computed by Algorithm 1.  If $\gvec{\alpha}$ is a basis for a subgroup of an abelian $p$-group $G$ and $\beta\in G$, we wish to compute a pair $(\evec{x},h)$, with $\evec{x}=\DL_{\gvec{\alpha}}(j+h,k,\beta^{p^h})$ and $h\ge 0$ minimal.  It may be that there is no $h<k-j$ for which such a pair exists, and in this case\footnote{Arguably, $h$ should be $\log_p|\beta|$ here (so that $\beta^{p^h}=\gvec{\alpha}^{p^h\evec{x}}$), but this is less convenient.} we let $h=k-j$ and $\evec{x}=0$.  When $\gvec{\alpha}$ generates a subgroup with exponent $p^m$, we use $\DL^*(\gvec{\alpha},\beta)$ to denote $\DL_{\gvec{\alpha}}^*(0,m,\beta)$.

\begin{algorithm}\label{algorithm:EDLp}
Let $\gvec{\alpha}$ be a basis for a subgroup of a finite abelian $p$-group $G$ and let $t\in\Z_{>0}$.
Given $\beta\in G$ and $0\le j<k$, compute $(\evec{x},h)=\DL_{\gvec{\alpha}}^*(j,k,\beta)$ as follows:
\end{algorithm}
\renewcommand\labelenumi{\theenumi.}
\begin{enumerate}
\item
If $k-j\le t$, compute $(\evec{x},h)\leftarrow\DL_{\gvec{\alpha}}^*(j,k,\beta)$ as a base case.  Return $(\evec{x},h)$.
\vspace{4pt}
\item
Choose integers $j_1,\ldots,j_w$ satisfying $j=j_1<j_2<\ldots<j_w<j_{w+1}=k$.
\vspace{4pt}
\item
Compute $\gamma_i=\beta^{p^{{j_i}-j}}$ for $i$ from 1 to $w$, and set $\evec{x}\leftarrow 0$.
\vspace{4pt}
\item
For $i$ from $w$ down to 1:
\renewcommand\labelenumii{\theenumii.}
\begin{enumerate}
\item
\vspace{3pt}
Recursively compute $(\evec{v},h)\leftarrow \DL_{\gvec{\alpha}}^*(j_i,j_{i+1},\gamma_i\gvec{\alpha}(j_i,k)^{-\evec{x}})$.
\item
\vspace{3pt}
Set $\evec{x}\leftarrow \evec{s}\evec{v} + \evec{x}$, where $\evec{s}=\evec{q}(j_i+h,j_{i+1})/\evec{q}(j_i+h,k)$.
\item
\vspace{3pt}
If $h>0$ then return $(\evec{x},j_i+h)$.
\end{enumerate}
\vspace{4pt}
\item 
Return $(\evec{x},0)$.
\end{enumerate}
\vspace{6pt}
\noindent
For $t=1$, the base case simply computes $\evec{x}=\DL(\gvec{\alpha},\beta)$ and returns $(\evec{x},0)$, or $(\evec{0},1)$ if a failure occurs.  When $t>1$, one applies Lemma \ref{lemma:EDLreduction} (if a lookup table is used, this means $O(\lg{t})$ table lookups and $O(t\lg{p})$ group operations).
\smallbreak

Aside from the computation of $h$ and the possibility of early termination, Algorithm 3 is essentially the same as Algorithm 1.  Indeed, assuming $t=1$, if $(\evec{x},h)$ is the output of Algorithm 3, the sequence of group operations performed by Algorithm 1 on input $\gvec{\alpha}^{\evec{x}}$ will be effectively identical (ignoring operations involving the identity).  Thus the complexity bounds in Proposition \ref{proposition:DLComplexity} and its corollary apply.

To verify the correctness of Algorithm 3, we first note that if $h=k-j$, then the first base case must have failed and the output $(\evec{0},h)$ is correct.  If $h<k-j$, then it follows from the correctness of Algorithm 1 that $\evec{x}=\DL_{\gvec{\alpha}}(j+h,k,\beta^{p^h})$.  It is only necessary to check that $h$ is minimal, but if not, the base case $\DL_{\evec{\alpha}}(j+h-1,j+h,\beta')$ would have succeeded and $h$ would be smaller.

%$When $\EDL(\gvec{\alpha},\beta) = (0,0)$ the algorithm is particularly efficient, requiring only a single base case computation in $H_{s-1}^s$.  If $r_s$ is the rank of $H_{s-1}^s$, the total cost is then
%$$O(\lg|H| + p^{r_s/2})$$
%group operations.  The case $\EDL(\gvec{\alpha},\beta) = (0,0)$ is significant because it implies that $\gvec{\alpha}\circ\beta=(\alpha_1,\ldots,\alpha_r,\beta)$ is a basis for the group $\langle\gvec{\alpha},\beta\rangle$.

We now explain how to construct a basis using Algorithm 3.  Let us start with a vector $\gvec{\alpha}$ consisting of a single element of $G$.  Clearly $\gvec{\alpha}$ is a basis for the cyclic subgroup it generates, and we would like to extend $\gvec{\alpha}$ to a basis for all of $G$ by adding elements to it one by one.  This will only be possible if our basis at each step generates a subgroup $H$ that is a factor of $G$ (meaning $G\cong H\times G/H$).  Some care is required, since $H$ need not be a factor of $G$, but let us first consider how to extend a basis.

Given a basis $\gvec{\alpha}$ for a subgroup $H$ of $G$, we say that $\gamma\in G$ is \emph{independent} of~$\gvec{\alpha}$ if the vector $\gvec{\alpha}\circ\gamma=(\alpha_1,\ldots,\alpha_r,\gamma)$ is a basis for $\langle \gvec{\alpha},\gamma\rangle$, and we write $\gamma\perp\gvec{\alpha}$.  The following lemma shows how and when one may use $\DL^*(\gvec{\alpha},\beta)$ to obtain such a $\gamma$.
\begin{lemma}\label{lemma:Independence}
Let $\gvec{\alpha}$ be a basis for a subgroup of a finite abelian $p$-group $G$, with $n_i=\log_p|\alpha_i|$, $m_0=\min n_i$, and $m=\max n_i$.  Let $\beta\in G$ and let $\gamma=\beta\gvec{\alpha}^{-\evec{x}}$, where  $(\gvec{x},h)=\DL^*(\gvec{\alpha},\beta)$.  The following hold:
\renewcommand\labelenumi{(\roman{enumi})}
\begin{enumerate}
\item
If $|\beta|\le p^m$, then $|\gamma|=p^h$.
\vspace{3pt}
\item
If $|\beta|\le p^m$ and $|\gamma| \le p^{m_0}$, then $\gamma\perp\gvec{\alpha}$.
\end{enumerate}
\end{lemma}
\begin{proof}
If $h<m$, then $\evec{x}=\DL_{\gvec{\alpha}}^*(h,m,\beta)$, and we have
$$\beta^{p^h}=\gvec{\alpha}(h,m)^{\evec{x}}=\gvec{\alpha}^{\evec{q}(h,m)\evec{x}}=\gvec{\alpha}^{p^h\evec{x}},$$
since $\evec{q}(h,m)$, as defined in (\ref{equation:alphajk}), has $q_i=p^{\max(h,n_i-m-h)}=p^h$.  It follows that $$\gamma^h=(\beta\gvec{\alpha}^{-\evec{x}})^{p^h}=\identity,$$
and this cannot hold for any $h'<h$, by the minimality of $h$.  Thus (i) holds when $h<m$. 
Now suppose $h=m$. Then $\evec{x}=\evec{0}$, $\gamma=\beta$, and $|\gamma|=|\beta|\ge p^h=p^m$.  If $|\beta|\le p^m$, then $|\gamma|=|\beta|=p^h$; thus (i) also holds when $h=m$.

To prove (ii), assume $|\beta|\le p^m|$ and $|\gamma|\le p^{m_0}$, and suppose $\gamma\perp\gvec{\alpha}$ does not hold.  Then there is a nontrivial relation of the form $\gamma^{p^j}=\gvec{\alpha}^{\evec{z}}$, for some $\evec{z}\in R_{\gvec{\alpha}}$ and $j < h$, since $|\gamma|=p^h$ by (i).  We claim that $\evec{z}$ is not divisible by $p^j$, since
$$\gamma^{p^j}=(\beta\gvec{\alpha}^{-\evec{x}})^{p^j} = \beta^{p^j}\gvec{\alpha}^{-p^j\evec{x}}=\gvec{\alpha}^{\evec{z}},$$
and if $p^j$ divides $\evec{z}$ we can set $\evec{v}=\evec{x}+\evec{z}/p^j$ to obtain
$$\beta^{p^j}=\gvec{\alpha}^{\evec{z}+p^j\evec{x}}=(\gvec{\alpha}^{\evec{v}})^{p^j}=\gvec{\alpha}^{\evec{q}(j,m)\evec{v}}=\gvec{\alpha}(j,m)^{\evec{v}},$$
which contradicts the minimality of $h$.  We now note that if $|\gamma|\le p^{m_0}$, then $\gamma^{p^j}$ has order at most $p^{m_0-j}$.  But $\gvec{\alpha}^{\evec{z}}$ has order greater than $p^{m_0-j}$, since some $z_i$ is not divisible by $p^j$, and therefore $|\alpha_i^{z_i}| > p^{n_i-j}\ge p^{m_0-j}$, yielding a contradiction.
\end{proof}

Lemma \ref{lemma:Independence} not only tells us how to find independent elements, it gives sufficient conditions to ensure that this is possible.  This yields a remarkably simple algorithm to construct a basis from a generating set $S$.

Start with $\evec{\alpha}$ consisting of a single element of $S$ with maximal order $p^m$.  Every $\beta\in S$ then satisfies $|\beta|\le p^m=p^{m_0}$, and we may use Algorithm \ref{algorithm:DLp} to compute an independent $\gamma=\beta\gvec{\alpha}^{-\evec{x}}$ for each $\beta$.  We can then choose one with maximal order to extend our basis $\evec{\alpha}$ and continue in this fashion until we have a basis spanning the entire group generated by $S$.

\begin{algorithm}\label{algorithm:pbasis}
Given a subset $S$ of a finite abelian $p$-group, the following algorithm computes a basis $\gvec{\alpha}$ for $G=\langle S\rangle$:
\end{algorithm}
\renewcommand\labelenumi{\theenumi.}
\begin{enumerate}
\item
Set $\gvec{\alpha}\leftarrow\emptyset$ and compute $h_i\leftarrow\log_p|\beta_i|$ for each $\beta_i\in S$.
\vspace{4pt}
\item
If every $h_i=0$, return $\gvec{\alpha}$.\\
Otherwise pick a maximal $h_i$, set $\gvec{\alpha}\leftarrow\gvec{\alpha}\circ \beta_i$, then $\beta_i\leftarrow\identity$ and $h_i\leftarrow 0$.
\vspace{4pt}
\item
For each $h_i>0$:
\renewcommand\labelenumii{\theenumii.}
\begin{enumerate}
\item
\vspace{3pt}
Compute $(\evec{x},h)\leftarrow \DL^*(\gvec{\alpha},\beta_i)$ using Algorithm \ref{algorithm:EDLp}.
\item
\vspace{3pt}
Set $\beta_i\leftarrow \beta_i\gvec{\alpha}^{-\evec{x}}$ and $h_i\leftarrow h$.
\end{enumerate}
\vspace{4pt}
\item
Go to step 2.
\end{enumerate}
\vspace{6pt}
\noindent
After step 1 we have (trivially) $\beta_i\perp\gvec{\alpha}$ for all $\beta_i$, and after step~2 we must have $h_i\le \min(\log_p|\alpha_i|)$ for all $h_i$.  By Lemma \ref{lemma:Independence}, these statements remain true after step~3, and at every step the algorithm ensures that $h_i=\log_p|\beta_i|$ and $\langle\gvec{\alpha},S\rangle = G$.  If every $h_i=0$, then $\langle S\rangle$ is trivial and $\gvec{\alpha}$ is a basis for $G$.  Some nonzero $h_i$ is set to zero each time step 2 is executed, so this eventually happens.  Note that when an element $\beta_i$ is appended to $\gvec{\alpha}$, its order $p^{h_i}$ is known, as desired.

\begin{proposition}\label{proposition:BasisComplexity}
Given a set $S$ that generates an abelian group $G$ of size $p^n$, exponent $p^m$, and rank $r$, there is a generic algorithm to compute a basis for $G$ using
$$\TB(S) \le c\Bigl(\frac{r\lg(m+1)}{\lg\lg(m+2)}\lg|G| + \frac{n}{r}p^{(r-1)/2}+ \bigl(|S|-r\bigr)\TDL(G)\Bigr),$$
group operations, where $\TDL(G)$ is as in Proposition \ref{proposition:DLComplexity} and $c$ is an absolute constant independent of $S$ and $G$.
\end{proposition}
\begin{proof}
We apply Algorithm \ref{algorithm:pbasis}, setting $t=1$ in Algorithm \ref{algorithm:EDLp}.  We assume that a table of all nontrivial $p^j$th powers of each element of $\gvec{\alpha}$ is maintained throughout, for a total cost of at most $2\lg|G|$ group operations (this table can also be made available to Algorithm \ref{algorithm:EDLp}, avoiding the need for any precomputation).  For each $\beta \in S$, computing $|\beta|$ in step 1 requires less than $2m\lg{p}$ group operations.  The cost of all the exponentiations in step 3b related to $\beta$ is bounded by $2\lg|G|$ (if we consider $\evec{x}=\DL(\gvec{\alpha},\beta)$ for the initial value of $\beta$ relative to the final basis $\gvec{\alpha}$, each base-$p$ digit of $x$ is ``cleared" in step 3b at most once).  The total cost of all steps other than 3a is thus $O(|S|\lg|G|)$ group operations, which is bounded by the sum of the first and last terms in the bound for $T_{B}(S)$, for a suitable constant $c$.

We now consider the cost of step 3a for those $\beta=\beta_i\in S$ for which $h_i$ is never chosen in step 2, meaning $\beta_i$ is never appended to $\gvec{\alpha}$.  There are exactly $|S|-r$ such~$\beta$.  For each base case that succeeds in some computation $\DL^*(\gvec{\alpha},\beta)$, the order of~$\beta$ is reduced by a factor of $p$ in step 3b, so there are at most $m$ successful base cases relevant to $\beta$ in the entire execution of Algorithm \ref{algorithm:pbasis}.  Ignoring the cost of reaching the first base case, and failed base cases, the successful part of all the $\DL^*(\gvec{\alpha},\beta)$ computations involving $\beta$ corresponds to a single computation $\DL(\gvec{\alpha},\beta)$ with respect to the final basis $\gvec{\alpha}$ for $G$, which we bound by $\TDL(G)$.

When computing $\DL^*(\gvec{\alpha},\beta)$, reaching the first base case involves exponentiating~$\beta$ (and precomputing $\gvec{\alpha}(j,k)$, but this was addressed above).  The order of $\beta$ is bounded by the order of the most recently added component $\alpha_i$ of $\gvec{\alpha}$, hence the total cost of all the initial exponentiations of $\beta$ is at most $\sum_{i=1}^r2\lg|\alpha_i|=2\lg|G|$, which is bounded by a constant factor of $\TDL(G)$.  Summing over the $|S|-r$ different values of~$\beta$ yields the term $(|S|-r)\TDL(G)$ in the bound for $\TB(S)$.

It remains to consider the cost of step 3a for the elements $\beta_1,\ldots,\beta_r$ that are at some point appended to $\gvec{\alpha}$.  With $n_i=\log_p|\alpha_i|$, we have $m=n_1\ge \cdots\ge n_r\ge 1$.  Define $m_i=n_i-n_{i+1}$ for $1\le i \le r-1$, and let $f(x)=x\lg(x+1)/\lg\lg(x+2)$.  Excluding base cases, the successful part of all computations $\DL^*(\gvec{\alpha},\beta_i)$ may be bounded, as in Proposition \ref{proposition:DLComplexity}, by a constant factor of
\begin{equation}\label{equation:BBound1}
\sum_{i=1}^{r-1}(r-i)f(m_i)\lg{p}\thickspace \le\thickspace (r-1)f(m)\lg{p}\thickspace<\thickspace \frac{r\lg(m+1)}{\lg\lg(m+2)}\lg|G|,
\end{equation}
where we have used $\sum f(m_i)\le f(m)$ for positive integers $m_i$ with $\sum m_i\le m$.  As above, the total cost of reaching the first base case in the computations $\DL^*(\gvec{\alpha},\beta_i)$ for $\beta_i$ is at most $2\lg|G|$,
which may be incorporated into (\ref{equation:BBound1}), yielding the first term of the bound for $\TB(S)$.

Finally, we consider the cost of the base cases occurring for $\beta_1,\ldots,\beta_r$.  In the $i$th iteration of step 3a there are $r-i$ elements $\beta_i$ which have yet to be appended to $\gvec{\alpha}$, and exactly one base case fails for each of these.  Thus we may bound the cost of all failed base cases by a constant factor of
\begin{equation}\label{equation:BBound2}
\sum_{i=1}^{r-1}(r-i)p^{i/2} < \frac{p^{(r+1)/2}}{(\sqrt{p}-1)^2} < 12p^{(r-1)/2}.
\end{equation}
For the successful base cases, let $r_j$ be the rank of the subgroup of $p^j$th powers in~$G$, as in Proposition \ref{proposition:DLComplexity}, so that $n=r_0+\cdots+r_{m-1}$.  For each $r_j$ we obtain a sum of the form (\ref{equation:BBound2}), and note that, as in Corollary \ref{corollary:DLComplexity}, we may bound the cost to within a constant factor by assuming the $r_j$ are all approximately equal to $r$.  In this case we have $m=\lceil n/r\rceil$, yielding the term $(n/r)p^{(r-1)/2}$ in the bound for $\TB(S)$, which also covers the failed base cases, for a suitable choice of $c$.
\end{proof}

If we are given a set $S$ of independent elements, Proposition \ref{proposition:BasisComplexity} implies that we can typically verify that $S$ is a basis for $G=\langle S \rangle$ more quickly than we can compute discrete logarithms in $G$.  In fact this is true whenever $|S|=r$, even if the elements of $S$ are not independent.  More generally, we have the following corollary.

\begin{corollary}\label{corollary:BasisComplexity}
Given a generating set $S$ for a finite abelian $p$-group $G$ of rank $r$, with $|S|=r+O(1)$, there is a generic algorithm to compute a basis for $G$ using
$$\TB(G)\medspace=\medspace O\left(\lg^{2+\epsilon}|G|\right)+O\bigl(\TDL(G)\bigr)\medspace=\medspace O(|G|^{1/2})$$
group operations.
\end{corollary}

When a generating set is not available, or when $|S|\gg r$, we may instead use a probabilistic algorithm to construct a basis from randomly sampled elements of $G$.  If $r$ is known (or bounded), Corollary \ref{corollary:BasisComplexity} can be applied to a randomly generated subset $S\subset G$ of size $r+t$ to obtain a generic Monte Carlo algorithm that is correct\footnote{This algorithm always outputs a basis for a subgroup $H$ of $G$, but it may be that $H<G$.} with probability at least $1-p^{-t}$.   This follows from the lemma below, whose proof can be found in \cite[Eq.~2]{Pomerance:GeneratingAbelianGroups} and also \cite[Lem.~4]{Acciaro:GeneratingGroups}.

\begin{lemma}\label{lemma:GeneratingProbability}
Let $G$ be a finite abelian $p$-group of rank $r$, and let $S$ be a sample of $s\ge r$ independent and uniformly distributed random elements of $G$.

Then $S$ generates $G$ with probability $\prod\limits_{j=s-r+1}^s\left(1-p^{-j}\right) > 1-p^{r-s}.$
\end{lemma}

In general, we do not know the rank of $G$, \emph{a priori}. Indeed, determining $r$ may be a reason for computing a basis.  In this situation we could apply Algorithm \ref{algorithm:pbasis} to progressively larger randomly generated sets $S$ until $|S|>r+t$, where $r$ is the rank of $\langle S\rangle$ and $t$ is a constant.  However, a more efficient approach is to simply select random $\beta\in G$, using the black box or via Lemma \ref{lemma:RandomElements} below, and attempt to use Lemma \ref{lemma:Independence} to extend the current basis.

This eliminates the loop in step 3 of Algorithm \ref{algorithm:pbasis}, but we must now address the situation where Lemma \ref{lemma:Independence} fails to apply ($|\beta|>p^m$ or $|\gamma|>p^{m_0}$). It may happen that the basis we have constructed cannot be extended to a basis for $G$, and in this case we need to backtrack.  Fortunately, this is easy to detect (and correct) and has negligible impact on the expected running time.

\begin{algorithm}\label{algorithm:pRandBasis}
Given a randomized black box for a finite abelian $p$-group $G$ and $t\in\Z_{>0}$, the following algorithm computes a basis $\gvec{\alpha}$ for a subgroup $H$ of $G$, where $H=G$ with probability at least $1-p^{-t}$:
\end{algorithm}
\renewcommand\labelenumi{\theenumi.}
\begin{enumerate}
\item
Set $s\leftarrow 0$.
Pick a random $\alpha_1\in G$ and set $\gvec{\alpha}\leftarrow(\alpha_1)$.
\vspace{4pt}
\item
If $s=t$, then return $\gvec{\alpha}$.
\vspace{4pt}
\item
Pick a random $\beta\in G$ and compute $(\evec{x},h)\leftarrow \DL^*(\gvec{\alpha},\beta)$.
\vspace{4pt}
\item
If $h=0$, then increment $s$ and go to step 2; otherwise set $\gamma\leftarrow \beta\gvec{\alpha}^{-\evec{x}}$.
\vspace{4pt}
\item
For each $\alpha_i$ with $|\alpha_i|<|\gamma|$, remove $\alpha_i$ from $\gvec{\alpha}$ and set $s\leftarrow 0$.
\vspace{4pt}
\item
Set $\gvec{\alpha}\leftarrow \gvec{\alpha}\circ\gamma$ and go to step 2.
\end{enumerate}
\vspace{6pt}
\noindent
The correctness of Algorithm \ref{algorithm:pRandBasis} depends on an easy corollary to Lemma \ref{lemma:Independence}.  If we let the (possibly empty) vector $\gvec{\alpha}'$ consist of those components of $\gvec{\alpha}$ that satisfy $|\alpha_i|\ge |\gamma|$, then $\gamma\perp \gvec{\alpha}'$ (the proof is the same).
It follows that after step 6, $\gvec{\alpha}$ is a basis for the subgroup it generates (this is obviously also true after step 1).  When the algorithm terminates, it has found $t$ (independent, uniformly distributed) random elements $\beta\in G$ that lie in $H=\langle\gvec{\alpha}\rangle$.  If $H$ is a proper subgroup of $G$, it must be smaller by a factor of at least $p$; the probability that $t$ random elements $\beta\in G$ all happen to lie in $H$ is then at most $p^{-t}$.

\begin{proposition}\label{proposition:pRandBasisComplexity}
Given a randomized black box for a finite abelian $p$-group $G$ of rank $r$, exponent $p^m$, and size $p^n$, and $t\in\Z_{>0}$, there is a probabilistic generic algorithm that computes a basis for a subgroup $H$ of $G$ using an expected
$$\TB^*(G)\medspace\le\medspace c\Bigl(\frac{r\lg(m+1)}{\lg\lg(m+2)}\lg|G| + \frac{n}{r}p^{(r-1)/2}\Bigr) + t\hspace{1pt}\TDL(G)\medspace =\medspace O(|G|^{1/2})$$
group operations, such that $H=G$ with probability at least $1-p^{-t}$.
The absolute constant $c$ is independent of both $t$ and $G$.
\end{proposition}
\begin{proof}
We apply Algorithm \ref{algorithm:pRandBasis}.  If it never backtracks (removes elements from $\gvec{\alpha}$ in step 5), the final basis $\gvec{\alpha}$ is obtained from the first $r$ random elements, and then $t$ discrete logarithms are computed using this basis.  In this case, the bound $\TB^*(G)$ follows from an argument similar to that used in the proof of Proposition \ref{proposition:BasisComplexity}, with $|S|=r$ (and a better constant factor).  We will show that the expected cost of Algorithm \ref{algorithm:pRandBasis} is within a constant factor of the cost arising in this ideal scenario.

Let $\gvec{\alpha}=(\alpha_1,\ldots,\alpha_r)$ be the final basis output by Algorithm \ref{algorithm:pRandBasis}, and note that $|\alpha_1|\ge|\alpha_2|\ge\cdots\ge|\alpha_r|$.  As the computation proceeds, for each $k$ from 1 to $r$, there is a stage $k$ where $\gvec{\alpha}_{k-1}=(\alpha_1,\ldots,\alpha_{k-1})$ is a (possibly empty) prefix of the final basis, and the algorithm is in the process of determining $\alpha_k$.  This may involve extending and then backtracking to the prefix $\gvec{\alpha}_{k-1}$ (several times, perhaps), but once $\alpha_k$ is determined, we have the prefix $\gvec{\alpha}_k$ and transition to stage $k+1$.  If no backtracking occurs, the algorithm completes stage 1 after step 1, and a single computation of $\DL^*(\gvec{\alpha}_k,\beta)$ is required for each stage $k > 1$.  From Proposition \ref{proposition:DLComplexity}, the cost of this computation may be bounded by
\begin{equation*}
S_k=c_1\sum_{i=1}^{k-1}\frac{\lg(n_i+1)}{\lg\lg(n_i+2)}\lg|\alpha_i|+c_1\sum_{j=0}^{m-1} p^{r_j/2} = A_k + B_k,
\end{equation*}
where $c_1$ is a constant, $n_i=\log_p|\alpha_i|$, and the ranks $r_j\le k-1$ are as in Proposition~\ref{proposition:DLComplexity}.  Let $A_k$ and $B_k$ denote the two sums in $S_k$, including the factor $c_1$.

We now consider the probability that the computation $\DL^*(\gvec{\alpha}_{k-1},\beta)$ completes stage $k$.  Let $\evec{z}$ be the discrete logarithm of $\beta$ relative to the final basis $\gvec{\alpha}$.  Provided that $z_k$ is not divisible by $p$, when $h$ is computed in step 3 we will have $h=n_k$, and compute $\gamma=\alpha_k$ in step 4, since no subsequent computation can yield an independent element of order greater than $n_k$ (since $n_j\le n_k$ for $j>k$).  Thus for each random $\beta\in G$ processed during stage~$k$, the probability that we do not complete stage~$k$ is at most $1/p$ (this is true for any extension of $\gvec{\alpha}_{k-1}$ arising during stage~$k$).  Conditioning on $w$, the number of random $\beta\in G$ processed during stage~$k$, the expected cost of stage~$k$ may be bounded by a sum of the form
\begin{align}\label{equation:pRandBasis1}\notag
T_k &\le (1-p^{-1})(A_k+B_k)+p^{-2}((1+2)A_k + (1+p^{1/2})B_k)+\ldots\\
T_k &\le \left(\frac{p-1}{p} + \sum_{w=2}^\infty \binom{w+1}{2}p^{-w}\right)A_k+\left(\frac{p-1}{p}+ \sum_{w=2}^\infty bp^{-w/2}\right)B_k,
\end{align}
where $b=1/(\sqrt{p}-1)$.  We have assumed here, as a worst case, that after processing each $\beta$ the current basis is extended by a $\gamma$ that maximizes the cost of subsequent discrete logarithm computations.  For each increment in $w$ we suppose that $|\langle\gvec{\alpha}\rangle|$ increases by a factor of $|\langle\gvec{\alpha}_{k-1}\rangle|$ (in fact, it increases by at most a factor of $|\alpha_{k-1}|$) and that every $r_j$ increases by 1.

The second sum in (\ref{equation:pRandBasis1}) is a geometric series, bounded by $b/(p-\sqrt{p})<5$. Summation by parts yields the identity
$$\sum_{w=1}^{\infty}\binom{w+1}{2}p^{-w}=\frac{p^2}{(p-1)^3},$$
allowing us to bound the first sum in (\ref{equation:pRandBasis1}) by 4.  Hence $T_k\le c_2S_k$ for a constant $c_2<6$, and the bound on $\TB^*(G)$ follows.  The correctness probability was addressed above, and clearly $c=c_1c_2$ is independent of $t$ and $G$.
\end{proof}

%As usual, we have overestimated the constant factors in the proof above.
In practice, the constant $c$ in Proposition \ref{proposition:pRandBasisComplexity} is quite small and $\TB^*(G)\approx t\TDL(G)$, even when $p=2$ (the worst case, as far as the constant factors are concerned).  When $\TDL(G)$ is dominated by $p^{r/2}$, the constant $t$ can be improved to $\sqrt{t}$ using a baby-steps giant-steps approach, as discussed in Section \ref{section:Performance}.

If we are given a bound $M$ satisfying $M\le |G| < pM$ (perhaps $M=|G|$), we can easily convert Algorithm \ref{algorithm:pRandBasis} from a Monte Carlo algorithm to a Las Vegas algorithm by replacing the test ``$s=t$" in step 2 with ``$|\langle\gvec{\alpha}\rangle|\ge M$" (note that $|\langle\gvec{\alpha}\rangle|=\prod|\alpha_i|$).

We now give a method to construct uniformly random elements of $G$ from a generating set $S$.  This is useful in general and allows us to apply Algorithm \ref{algorithm:pRandBasis} to a generating set $S$, which may be faster than using Algorithm \ref{algorithm:pbasis} when $|S|\gg r$.

\begin{lemma}\label{lemma:RandomElements}
Given a generating set $S$ for a finite abelian group $G$ with exponent $p^m=E$, and $t\in\Z_{>0}$, there is a generic algorithm to compute $t$ independent, uniformly random elements of $G$ using
$$T_R(S,t)\le(3\lg{E})|S|+2t\left(\frac{\lceil\lg(E+1)\rceil}{\lceil\lg\lg(E+2)\rceil}+1\right)|S|$$
group operations and storage for at most $t+\lg{E}$ group elements.
\end{lemma}
\begin{proof}
We represent $S$ as a vector $\gvec{\gamma}=(\gamma_1,\ldots,\gamma_s)$.
To construct random elements $\beta_1,\ldots,\beta_t$, first set each $\beta_j$ to $\identity$.  Then, for each $\gamma_i$, compute $|\gamma_i|=p^{n_i}$, select $t$ uniformly random integers $z_{i,j}\in[0,p^{n_i})$, and set $\beta_j\leftarrow \beta_j\gamma_i^{z_{i,j}}$ for each $j$.  We assume all the $z_{i,j}$ are chosen independently.

The cost of computing $|\gamma_i|$ is at most $2m\lg{p}=2\lg{E}$ group operations.  By Yao's Theorem, the cost of $t$ exponentiations of the common base $\gamma_i$ is at most
$$\lg{E}+c\hspace{1pt}t\left(\frac{\lceil\lg(E+1)\rceil}{\lceil\lg\lg(E+2)\rceil}\right)$$
group operations, where $c\le 2$.  Accounting for the multiplication by $\beta_j$ and summing over the $\gamma_i$ yields the bound $T_R(S,t)$.  We only need to store the $\beta_j$ and at most $\lg{E}$ powers of a single $\gamma_i$ at each step (we don't count the size of the input set $S$, since we only access one element of $S$ at a time).

Clearly the $\beta_j$ are independent; we must show that each is uniformly distributed over~$G$.
Let $H=\ZZ/p^{n_1}\ZZ\times\cdots\times\ZZ/p^{n_s}\ZZ$.  The map $\varphi:H\to G$ that sends $\evec{z}$ to $\gvec{\gamma}^\evec{z}$ is a surjective group homomorphism, and we have $\beta_j = \varphi(\evec{z})$, where $\evec{z}=(z_{1,j},z_{2,j}\ldots,z_{s,j})$ is uniformly distributed over $H$.  As each coset of $\ker\varphi$ has the same size, it follows that $\beta_j$ is uniformly distributed over $G\cong H/\ker\varphi$.
\end{proof}

In practice, we may wish to generate random elements ``on demand", without knowing $t$.  We can generate random elements in small batches of size $t\approx \lg\lg(E+2)$ to effectively achieve the same result.  If $S$ is reasonably small, the first term of $T_R(S,t)$ may be treated as a precomputation and need not be repeated.

Provided that $|S|=O(|G|^{1/2-\epsilon})$, we may apply Lemma \ref{lemma:RandomElements} and Proposition \ref{proposition:pRandBasisComplexity} to compute a basis for $G=\langle S\rangle$, with high probability, using $O(|G|^{1/2})$ group operations.  By contrast, Algorithm \ref{algorithm:pbasis} uses $O(|S||G|^{1/2})$ group operations when $S$ is large, as does the algorithm of Buchmann and Schmidt \cite{Buchmann:GroupStructure}.  However, we note that both of these algorithms are (or can be made) deterministic.

\section{Constructing a basis in the general case}\label{section:GeneralBasis}

We now suppose that $G$ is an arbitrary finite abelian group.  If we know the exponent of $G$, call it $\lambda(G)$, and its factorization into prime powers, we can easily reduce the computation of a basis for $G$ to the case already considered.  In fact, it suffices to know any reasonably small multiple $N$ of $\lambda(G)$, including $N=|G|$.  Factoring $N$ does not require any group operations, and it is, in any event, a much easier problem than computing $\lambda(G)$ in a generic group, hence we ignore this cost.\footnote{We have subexponential-time probabilistic algorithms for factoring versus exponential lower bounds for computing the group exponent with a probabilistic generic algorithm \cite{Babai:PolynomialTime}. Most deterministic factoring algorithms are already faster than the $\Omega(N^{1/3})$ lower bound of \cite[Thm.~2.3]{Sutherland:Thesis}.}

As shown in the author's thesis, $\lambda(G)$ can be computed using $o(|G|^{1/2})$ group operations  \cite{Sutherland:Thesis}.  This bound is strictly dominated by the worst-case complexity of both the algorithms presented in the previous section, allowing us to extend our complexity bounds for abelian $p$-groups to the general case.  The basic facts needed for the reduction are given by the following lemma.

\begin{lemma}\label{lemma:SubSample}
Let $G$ be a finite abelian group and let $N$ be a multiple of $\lambda(G)$.  Let $p_1,\ldots,p_k$ be distinct primes dividing $N$, and let $G_{p_i}$ be the Sylow $p_i$-subgroup of $G$.
\renewcommand\labelenumi{(\roman{enumi})}
\begin{enumerate}
\item
Given a generating set $S$ for $G$, one can compute generating sets $S_1,\ldots,S_k$ for $G_{p_1},\ldots,G_{p_k}$, each of size $|S|$, using $O\left(|S|\lg^{1+\epsilon}N\right)$ group operations.
\vspace{3pt}
\item
Given a uniformly distributed random $\beta\in G$, one can compute elements $\beta_1,\ldots,\beta_k$ uniformly distributed over the groups $G_{p_1},\ldots,G_{p_k}$ $($respectively$)$, using $O\left(\lg^{1+\epsilon}N\right)$ group operations.
\end{enumerate}
\end{lemma}
\begin{proof}
Let $N_i$ be the largest divisor of $N$ relatively prime to $p_i$.  Given $\beta\in S$, or a random $\beta\in G$, we compute $\beta_1=\beta^{N_1},\ldots,\beta_k=\beta^{N_k}$ with either Algorithm~7.3 or Algorithm~7.4 of \cite{Sutherland:Thesis}, using $O(\lg^{1+\epsilon}N)$ group operations.\footnote{Algorithm 7.3 is due to Celler and Leedham-Green \cite{Celler:GLOrder}.}

The map $\phi_i:G\to G_{p_i}$ sending $\beta$ to $\beta^{N_i}$ is a surjective group homomorphism, invertible on $G_{p_i}\subset G$.  Thus if $S$ generates $G$, then $S_i=\phi_i(S)$ generates $G_{p_i}$, which proves (i).  If $\beta$ is uniformly distributed over $G$, then $\phi_i(\beta)$ is uniformly distributed over $G_{p_i}$, proving (ii).
\end{proof}

We now extend Propositions \ref{proposition:BasisComplexity} and \ref{proposition:pRandBasisComplexity} to arbitrary finite abelian groups.

\begin{proposition}\label{proposition:basis}
Let $G$ be a finite abelian group whose nontrivial Sylow subgroups are $G_{p_1},\ldots,G_{p_k}$, and suppose that the exponent $($resp.~order$)$ of $G$ is given.  Let $S$ be a generating set for $G$, with $S_i$ as in Lemma \ref{lemma:SubSample}.
\vspace{3pt}
\renewcommand\labelenumi{(\roman{enumi})}
\begin{enumerate}
\item
There is a generic algorithm to compute a basis for $G$ which uses
$$O\bigl(|S|\lg^{1+\epsilon}|G|\bigr)\medspace+\medspace\sum \TB(S_i)$$
group operations, where $\TB(S_i)$ is bounded as in Proposition \ref{proposition:BasisComplexity}.
\vspace{8pt}
\item
Given a randomized black box for $G$, there is a Monte Carlo $($resp.~Las Vegas$)$ generic algorithm to compute a basis for $G$ using an expected
$$O\bigl(\lg^{2+\epsilon}|G|\bigr)\medspace+\medspace\sum \TB^*(G_{p_i})= O(|G|^{1/2})$$
group operations, where $\TB^*(G_{p_i})$ is bounded as in Proposition \ref{proposition:pRandBasisComplexity}.
\end{enumerate}
\end{proposition}
\begin{proof}
(i) is immediate from Lemma \ref{lemma:SubSample} and Proposition \ref{proposition:BasisComplexity}.  (ii)
follows similarly from Proposition \ref{proposition:pRandBasisComplexity} and the comments following, using the bound
$$\sum |G_{p_i}|^{1/2}\medspace\le\medspace \frac{3}{2}\prod |G_{p_i}|^{1/2} \medspace=\medspace \frac{3}{2}|G|^{1/2}$$
from Lemma \ref{lemma:prodsum}.
\end{proof}

\begin{corollary}\label{corollary:basis}
Given a randomized black box for a finite abelian group $G$, there is a Monte Carlo algorithm to compute a basis for $G$ using $O(|G|^{1/2})$ group operations.
\end{corollary}
\begin{proof}
Algorithm 8.1 of \cite{Sutherland:Thesis} computes $N=\lambda(G)$ with high probability and uses $o(\sqrt{N})$ group operations, assuming Algorithms 5.1 and 5.2 of \cite{Sutherland:Thesis} are used for order computations.  The corollary then follows from (ii) of Proposition~\ref{proposition:basis}.
\end{proof}

If we are given a generating set $S$ with $|S|=O(|G|^{1/2-\epsilon})$, we may apply Lemma~\ref{lemma:RandomElements} to obtain an analogous corollary.

The space required by the algorithms of Proposition \ref{proposition:basis} and Corollary \ref{corollary:basis} can be made quite small, polynomial in $\lg|G|$, using algorithms based on Pollard's rho method to handle the base cases of the discrete logarithm computations and applying the search used in Algorithm 5.1 of \cite{Sutherland:Thesis}.  If this is done, the complexity bound for computing $\lambda(G)$ increases to $O(N^{1/2})$ (but will typically be better than this).

It is not necessary to use a particularly fast algorithm to compute $\lambda(G)$ in order to prove Corollary \ref{corollary:basis}; any $O(N^{1/2})$ algorithm suffices.  However, the time to compute a basis for $G$ is often much less then $|G|^{1/2}$ group operations, as the worst case may arise rarely in practice.  Applying Algorithms 5.1 and 5.2 of \cite{Sutherland:Thesis} yields considerable improvement in many cases.\footnote{For example, Teske reports computing a basis for the ideal class group $G$ of $\mathbb{Q}[\sqrt{D}]$, with $D=-4(10^{30}+1)$, using $243,207,644\approx 7.1|G|^{1/2}$ group operations \cite{Teske:GroupStructure}.  In \cite{Sutherland:Thesis}, a basis for $G$ is computed using $250,277\approx 2.4|G|^{1/3}$ group operations.}

These comments are especially relevant when one only wishes to compute a basis for a particular Sylow $p$-subgroup $H$ of $G$ (perhaps as a prelude to extracting $p$th roots in $G$).  Once we have computed $\lambda(G)$, we can compute a basis for any of~$G$'s Sylow subgroups with a running time that typically depends only on the size and shape of the subgroup of interest, not on $G$.  The following proposition follows immediately from Lemma \ref{lemma:SubSample} and Proposition \ref{proposition:pRandBasisComplexity}.

\begin{proposition}
Let $H$ be a Sylow $p$-subgroup of a finite abelian group $G$.  Given a multiple $N$ of the exponent of $G$ and a randomized black box for $G$, there is a probabilistic generic algorithm to compute a basis for $H$ using
$$O\bigl(r\lg^{1+\epsilon}{N}\bigr)+ \TB^*(H)= O(r\lg^{1+\epsilon}{N}+|H|^{1/2})$$
group operations, where $r$ is the rank of $H$.
\end{proposition}

\section{Performance results}\label{section:Performance}

We tested the new algorithms on abelian $p$-groups of various sizes and shapes in order to assess their performance.  As in previous sections, $G$ is an abelian group of size $p^n$, exponent $p^m$, and rank $r$, whose shape is given by a partition of $n$ into~$r$ parts, with largest part $m$.

Here we present results for $p=2$, as this permits the greatest variation in the other parameters, and also because the Sylow 2-subgroup is of particular interest in many applications.  Results for other small primes are similar.  When $p$ is large, the results are not as interesting: $n$, $r$, and $m$ are all necessarily small, and the computation is dominated by the discrete logarithms computed in the base cases, whose $\Theta(p^{r/2})$ performance is well understood.

Our tests in $p$-groups used a black box which represents each cyclic factor of $G$ using integers mod $p^{n_i}$.
This is a convenient but arbitrary choice.  Identical results are obtained for any black box implementation, since the algorithms are generic.  Our performance metric counts group operations (multiplications and inversions) and does not depend on the speed of the black box or the computing platform.\footnote{Thus the performance results reported here are not impacted by Moore's law.}

To compute discrete logarithms in the base cases, we used Shanks' baby-steps giant-steps algorithm \cite{Shanks:BabyGiant} extended to handle products of cyclic groups.  Rather than the lexicographic ordering used by Algorithm 9.3 of \cite{Sutherland:Thesis}, we instead compute a Gray code \cite{Knuth:AOPIV2} when enumerating steps, always using one group operation per step (this is especially useful for small $p$, saving up to a factor of 2).  A more significant optimization available with Shanks' method is the ability to perform $k$ discrete logarithms in a group of size $N$ using $2\sqrt{kN}$ (rather than $2k\sqrt{N}$) group operations by storing $\sqrt{kN}$ baby steps in a lookup table and then taking $\sqrt{N/k}$ giant steps as each of the $k$ discrete logarithms is computed.\footnote{One uses $\sqrt{kN/2}$ baby steps to optimize the expected case, assuming $\beta\in\langle\gvec{\alpha}\rangle$.}

This optimization is useful in Algorithms \ref{algorithm:DLp} and \ref{algorithm:EDLp}, even for a single discrete logarithm computation, as there may be many base cases in the same subgroup.  It is even more useful in the context of Algorithms \ref{algorithm:pbasis} and \ref{algorithm:pRandBasis}, as several calls to Algorithm \ref{algorithm:EDLp} may use the same basis.  In the bound for $\TDL(G)$ in Corollary \ref{corollary:DLComplexity}, this effectively replaces the factor $n/r$ by $\sqrt{n/r}$.  When the rank-dependent terms in $\TDL(G)$ dominate sufficiently, the bounds in Propositions \ref{proposition:BasisComplexity} and \ref{proposition:pRandBasisComplexity} can be improved by replacing  $|S|-r$ with $\sqrt{|S|-r|}$ and $t$ with $\sqrt{t}$ (respectively).
\bigbreak

Table 1 lists group operation counts for Algorithm \ref{algorithm:DLp} when computing discrete logarithms in $2$-groups of rectangular shape, corresponding to partitions of $n$ into $r$ parts, all of size $m=n/r$.  Each entry is an average over 100 computations of $\DL(\gvec{\alpha},\beta)$ for a random $\beta\in G$.  Precomputation was optimized for a single discrete logarithm (repeated for each $\beta$) and these costs are included in Table \ref{table:DLTests}.  Reusing precomputed values can improve performance significantly over the figures given here, particularly when additional space is used, as in \cite{Sutherland:AbelianRoots}.

Algorithm 1 used the parameter $t=\lfloor(\lg{n}-1)/r\rfloor$ in these tests, which was near optimal in most cases.  The optimal choice of $w$ is slightly less than that used in the proof of Proposition \ref{proposition:DLComplexity}, as the average size of the exponents is smaller than the bound used there.  For each entry in Table \ref{table:DLTests}, if one computes the bound on $\TDL(G)$ given by Proposition \ref{proposition:DLComplexity}, we find the constant $c$ close to 1 in most cases (never more than 1.5).  In the first four columns of Table \ref{table:DLTests}, the counts are dominated by the exponent-dependent terms of $\TDL(G)$, explaining the initially decreasing costs as $r$ increases for a fixed value of $n$ ($r$ is larger, but $m=n/r$ is smaller).  

\begin{table}
\begin{center}
\begin{tabular}{@{}r|rrrrrr@{}}
%\toprule
$n\backslash r$&${\bf 1}$ & ${\bf 2}$ &${\bf 4}$&${\bf 8}$& ${\bf 16}$& ${\bf 32}$\\
\midrule
  {\bf 32} &    113 &     89 &     76 &     94 &    669 &   97936\\
  {\bf 64} &    261 &    204 &    172 &    194 &    853 &  163750\\
 {\bf 128} &    591 &    455 &    380 &    370 &   1501 &  197518\\
 {\bf 256} &   1268 &   1021 &    833 &    760 &   2065 &  328839\\
 {\bf 512} &   2718 &   2165 &   1770 &   1607 &   3760 &  395187\\
 {\bf 1024}&   5949 &   3931 &   3755 &   4601 &   5745 & 657965\\
\bottomrule
\end{tabular}
\vspace{4pt}
\caption{Group operations to compute $DL(\gvec{\alpha},\beta)$ in $G\cong\left(\ZZ/2^{n/r}\ZZ\right)^r$.}\label{table:DLTests}
\end{center}
\end{table}

Table \ref{table:DLComparison} compares the performance of Algorithm \ref{algorithm:DLp} to Teske's generalization of the Pohlig--Hellman algorithm on groups of order $2^{256}$ with a variety of different shapes. The baby-steps giant-steps optimization mentioned above is also applicable to Teske's algorithm (with even greater benefit), and we applied this optimization to both algorithms. The figures in Table \ref{table:DLComparison} reflect averages over 100 computations of $DL(\gvec{\alpha},\beta)$ for random $\beta\in\langle\gvec{\alpha}\rangle$.

\begin{table}
\begin{center}
\begin{tabular}{@{}llrr@{}}
%\toprule
Group&Structure & Pohlig--Hellman--Teske & Algorithm 1\\
\midrule
$G_1$&$256$& 32862&1268\\
$G_2$&$128\cdot 64\cdot 32\cdot 16\cdot 8\cdot 4\cdot 2\cdot 1^2$& 8736&1095\\
$G_3$&$16^{16}$&2065&1036\\
$G_4$&$26\cdot 22\cdot 21\cdot 20\cdots 3\cdot 2\cdot 1$&17610&6647\\
$G_5$&$128\cdot 32^2\cdot 8^4\cdot 2^8\cdot 1^{16}$&534953&84047\\
$G_6$&$226\cdot 1^{30}$&1075172&81942\\
\bottomrule
\end{tabular}
\vspace{4pt}
\caption{Computing discrete logarithms in groups of order $2^{256}$.}\label{table:DLComparison}
%\begin{minipage}{1.0\linewidth}
\vspace{-6pt}
\small
The notation $a\cdot b^c$ indicates the group $\ZZ/2^a\ZZ\times\left(\ZZ/2^b\ZZ\right)^c$.
\normalsize
%\end{minipage}
\end{center}
\end{table}

%None of the standard $O(\sqrt{N})$ algorithms can feasibly compute a discrete logarithm in any of these groups, regardless of shape.  By contrast, both Algorithm \ref{algorithm:DLp} and Teske's algorithm can be practically applied to 2-groups that are much larger and of greater rank (up to 60 or so) than those listed in Table \ref{table:DLComparison}.

Two advantages of Algorithm \ref{algorithm:DLp} are apparent in Table \ref{table:DLComparison}.  In the first two rows, the complexity is dominated by~$m$, and Algorithm \ref{algorithm:DLp} has a nearly linear dependence on~$m$, versus a quadratic dependence in the algorithms of Pohlig--Hellman and Teske.  In the last two rows, the complexity is dominated by $p^{r/2}$.  Algorithm~1 computes just one base case in a subgroup of size $p^r$, due to the shapes of the groups, while Teske's algorithm computes $m$ base case in a subgroup of size $p^r$ (using $O(m^{1/2}p^{r/2})$ group operations, thanks to the baby-steps giant-steps optimization).

Table \ref{table:Basis} presents performance results for Algorithms \ref{algorithm:pbasis} and \ref{algorithm:pRandBasis} when used to construct a basis for four of the groups listed in Table \ref{table:DLComparison}.  The group operation counts are averages over 100 tests.  The first four rows list results for Algorithm \ref{algorithm:pbasis} when given a random generating set $S$ of size $r+t$.  The case $t=0$ is of interest because it covers the situation where $S$ is itself a basis, hence it may function as a basis verification procedure.  The costs in this case are comparable to the cost of computing a single discrete logarithm in the group generated by $S$ (this improves for $p>2$).

The last four rows of Table \ref{table:Basis} give corresponding results for Algorithm \ref{algorithm:pRandBasis} using a randomized black box.  In the first row for Algorithm \ref{algorithm:pRandBasis}, the algorithm is given the order of the group and runs as a Las Vegas algorithm, terminating only when it has found a basis for the entire group.  In the remaining rows, Algorithm \ref{algorithm:pRandBasis} is used as a Monte Carlo algorithm, correct with probability at least $1-p^{-t}$.
\suppressfloats[t]
\begin{table}
\begin{center}
\begin{tabular}{@{}lrrrrr@{}}
%\toprule
&$\qquad t$&$\qquad G_2$&$\qquad G_3$&$\qquad G_4$&$\qquad G_5$\\
\midrule
Algorithm \ref{algorithm:pbasis} &0&897&1739&9231&169633\\
											&20&27077&15383&50528&406102\\
                                 &40&45741&24946&71752&586501\\\vspace{3pt}
                                 &80&82921&44337&111451&788065\\
Algorithm \ref{algorithm:pRandBasis}&-&12727&2770&49219&372876\\
												&20&27725&15027&68362&494345\\
                                 	&40&44137&26066&79950&587645\\
                                 	&80&76054&40843&109257&936478\\
\bottomrule
\end{tabular}
\vspace{4pt}
\caption{Computing a basis for groups of order $2^{256}$.}\label{table:Basis}
\end{center}
\end{table}
\section{Acknowledgments}
The author would like to thank David Harvey for his extensive feedback on an early draft of this paper, and the referee, for improving the proof of Lemma~6.

\section{Appendix}
The inequality below is elementary and surely known.  Lacking a suitable reference, we provide a short proof here.
\begin{lemma}
For any real number $a>1$ there is a constant $c\le a/e^{1+\ln\ln{a}}$ such that for all real numbers $x_1,\ldots,x_n\ge a$ $($and any $n)$,
$$\sum x_i\medspace\le\medspace c\prod x_i.$$
\end{lemma}
\begin{proof}
We assume $x_1\le \cdots\le x_n$.  If we fix $\sum x_i$, we can only decrease $\prod x_i$ by supposing $x_{n-1}=a$, since if $x_{n-1}=a+\delta$, we have
$$x_{n-1}x_n = ax_n+\delta x_n\le ax_n+\delta a = a(x_n+\delta).$$
We now assume $x_1=\cdots=x_{n-1}=a$ and $x_n = a+\delta$ with $\delta\ge 0$.  Since
$$f(\delta)\medspace=\medspace\sum x_i/\prod x_i\medspace=\medspace\frac{(n-1)a+\delta}{a^{n-1}\delta}$$
is a decreasing function of $\delta$, we maximize $\sum x_i /\prod x_i$ by assuming $x_n=a$ as well.

Thus it suffices to consider the case $\sum x_i/\prod x_i=na/a^n$, and we now view $g(n)=na/a^n$ as a function of a real variable $n$, which is maximized by $n=1/\ln{a}$.  Therefore, we may bound $\sum x_i/\prod x_i$ by $(1/\ln{a})/a^{1/\ln{a}} = a/e^{1+\ln\ln{a}},$
and the lemma follows.
\end{proof}
The bound on $c$ given in the lemma is not necessarily tight, since $n$ must be an integer.  If we note that $g(n)=na/a^n$ is increasing for $n<1/\ln{a}$ and decreasing for $n>1/\ln{a}$, it follows that the best possible $c$ is
\begin{equation}\label{equation:app1}
c = \min\left(g\left(\left\lfloor\frac{1}{\ln{a}}\right\rfloor\right), g\left(\left\lceil\frac{1}{\ln{a}}\right\rceil\right)\right).
\end{equation}
Applying (\ref{equation:app1}) with $a=\sqrt{2}$, we obtain the following lemma.
\begin{lemma}\label{lemma:prodsum}
For any integers $x_1,\ldots,x_n>1$ we have $\sum \sqrt{x_i}\medspace\le\medspace \frac{3}{2}\prod\sqrt{x_i}.$
\end{lemma}

\bibliographystyle{amsplain}
%\bibliography{../general}

\begin{thebibliography}{10}

\bibitem{Acciaro:GeneratingGroups}
Vincenzo Acciaro, \emph{The probability of generating some common families of
  finite groups}, Utilitas Mathematica \textbf{49} (1996), 243--254.

\bibitem{Adleman:Roots}
Leonard~M. Adleman, Kenneth Manders, and Gary~L. Miller, \emph{On taking roots
  in finite fields}, Proceedings of the 18th IEEE Symposium on Foundations of
  Computer Science, 1977, pp.~175--178.

\bibitem{Babai:PolynomialTime}
L{\'a}szl{\'o} Babai and Robert Beals, \emph{A polynomial-time theory of
  black-box groups {I}.}, Groups St. Andrews 1997 in Bath, I, London
  Mathematical Society Lecture Notes Series, vol. 260, Cambridge University
  Press, 1999, pp.~30--64.

\bibitem{Bernstein:SquareRoot}
Daniel~J. Bernstein, \emph{Faster square roots in annoying finite fields},
  \url{http://cr.yp.to/papers/sqroot.pdf}, 2001.

\bibitem{Bernstein:Pippenger}
\bysame, \emph{Pippenger's exponentiation algorithm},
  \url{http://cr.yp.to/papers/pippenger.pdf}, 2001.

\bibitem{BGMW:FastExponentiation}
Ernest~F. Brickell, Daniel~M. Gordon, Kevin~S. McCurley, and David~B. Wilson,
  \emph{Fast exponentiation with precomputation}, Advances in
  Cryptology--EUROCRYPT '92, Lecture Notes in Computer Science, vol. 658,
  Springer-Verlag, 1992, pp.~200--207.

\bibitem{Buchmann:BabyGiant}
Johannes Buchmann, Michael~J. {Jacobson, Jr.}, and Edlyn Teske, \emph{On some
  computational problems in finite abelian groups}, Mathematics of Computation
  \textbf{66} (1997), 1663--1687.

\bibitem{Buchmann:GroupStructure}
Johannes Buchmann and Arthur Schmidt, \emph{Computing the structure of a finite
  abelian group}, Mathematics of Computation \textbf{74} (2005), 2017--2026.

\bibitem{Buchmann:BinaryQuadraticForms}
Johannes Buchmann and Ulrich Vollmer, \emph{Binary quadratic forms: an
  algorithmic approach}, Algorithms and Computations in Mathematics, vol.~20,
  Springer, 2007.

\bibitem{Celler:GLOrder}
Frank Celler and C.~R. Leedham-Green, \emph{Calculating the order of an
  invertible matrix}, Groups and Computation {II}, DIMACS Series in Discrete
  Mathematics and Theoretical Computer Science, vol.~28, American Mathematical
  Society, 1997, pp.~55--60.

\bibitem{Cohen:CANT}
Henri Cohen, \emph{A course in computational algebraic number theory},
  Springer, 1996.

\bibitem{Gordon:Survey}
Daniel~M. Gordon, \emph{A survey of fast exponentiation methods}, Journal of
  Algorithms \textbf{27} (1998), 129--146.

\bibitem{Knuth:AOPIV2}
Donald~E. Knuth, \emph{The art of computer programming, volume {IV}, fascicle
  2: Generating all tuples and permutations}, Addison-Wesley, 2005.

\bibitem{Lim:Exponentiation}
Chae~Hoon Lim and Pil~Joong Lee, \emph{More flexible exponentiation with
  precomputation}, Advances in Cryptology--{CRYPTO '94}, Lecture Notes in
  Computer Science, vol. 839, Springer, 1994, pp.~95--107.

\bibitem{McCurly:DiscreteLogarithm}
Kevin~S. McCurley, \emph{The discrete logarithm problem}, Cryptography and
  Computational Number Theory (C. Pomerance, ed.), Proceedings of Symposia in
  Applied Mathematics, vol.~42, American Mathematical Society, 1990, p.~4974.

\bibitem{Menezes:Handbook}
Alfred~J. Menezes, Paul~C. van Oorschot, and Scott~A. Vanstone, \emph{Handbook
  of applied cryptography}, CRC Press, 1997, revised reprint.

\bibitem{Odlyzko:DiscretLogarithms}
Andrew Odlyzko, \emph{Discrete logarithms: The past and the future}, Designs,
  Codes, and Cryptography \textbf{19} (2000), 129--145.

\bibitem{Pippenger:PowersPrelim}
Nicholas Pippenger, \emph{On the evaluation of powers and related problems
  (preliminary version)}, 17th Annual Symposium on Foundations of Computer
  Science, IEEE, 1976, pp.~258--263.

\bibitem{Pohlig:DiscreteLog}
Stephen~C. Pohlig and Martin~E. Hellman, \emph{An improved algorithm for
  computing logarithms over ${GF}(p)$ and its cryptographic significance}, IEEE
  Transactions on Information Theory \textbf{24} (1978), 106--110.

\bibitem{Pollard:RhoDL}
John~M. Pollard, \emph{Monte {C}arlo methods for index computations mod $p$},
  Mathematics of Computation \textbf{32} (1978), 918--924.

\bibitem{Pomerance:GeneratingAbelianGroups}
Carl Pomerance, \emph{The expected number of random elements to generate a
  finite abelian group}, Periodica Mathematica Hungarica \textbf{43} (2001),
  191--198.

\bibitem{Shanks:BabyGiant}
Donald Shanks, \emph{Class number, a theory of factorization and genera},
  Analytic Number Theory, Proceedings of Symposia on Pure Mathematics, vol.~20,
  American Mathematical Society, 1971, pp.~415--440.

\bibitem{Shanks:FiveAlgorithms}
\bysame, \emph{Five number-theoretic algorithms}, Proceedings of the 2nd
  Manitoba Conference on Numerical Mathematics, 1972, pp.~51--70.

\bibitem{Shoup:DLLowerBound}
Victor Shoup, \emph{Lower bounds for discrete logarithms and related problems},
  Advances in Cryptology--EUROCRYPT '97, Lecture Notes in Computer Science,
  vol. 1233, Springer-Verlag, 1997, revised version, pp.~256--266.

\bibitem{Shoup:NumberTheoryAlgebra}
\bysame, \emph{A computational introduction to number theory and algebra},
  Cambridge University Press, 2005.

\bibitem{Sutherland:Thesis}
Andrew~V. Sutherland, \emph{Order computations in generic groups}, {P}h{D}
  thesis, MIT, 2007,
  \url{http://groups.csail.mit.edu/cis/theses/sutherland-phd.pdf}.

\bibitem{Sutherland:AbelianRoots}
\bysame, \emph{Extracting roots in finite abelian groups}, 2008, preprint.

\bibitem{Teske:GroupStructure}
Edlyn Teske, \emph{A space efficient algorithm for group structure
  computation}, Mathematics of Computation \textbf{67} (1998), 1637--1663.

\bibitem{Teske:SpeedingRho}
\bysame, \emph{Speeding up {P}ollard's rho method for computing discrete
  logarithms}, Algorithmic Number Theory Symposium--{ANTS III}, Lecture Notes
  in Computer Science, vol. 1423, Springer-Verlag, 1998, pp.~541--554.

\bibitem{Teske:PohligHellmanStructure}
\bysame, \emph{The {Pohlig-Hellman} method generalized for group structure
  computation}, Journal of Symbolic Computation \textbf{27} (1999), 521--534.

\bibitem{Tonelli:SquareRoot}
Alberto Tonelli, \emph{Bemerkung \"uber die {A}ufl\"{o}sung quadratischer
  {C}ongruenzen}, G\"ottinger Nachrichten (1891), 344--346.

\end{thebibliography}
\providecommand{\bysame}{\leavevmode\hbox to3em{\hrulefill}\thinspace}
\providecommand{\MR}{\relax\ifhmode\unskip\space\fi MR }
\providecommand{\MRhref}[2]{%
  \href{http://www.ams.org/mathscinet-getitem?mr=#1}{#2}
}
\providecommand{\href}[2]{#2}

\end{document}